\documentclass[preprint]{elsarticle}       

\usepackage{ulem}
\usepackage{lineno,hyperref}
\usepackage{rotating}
\usepackage{makeidx}
\usepackage{float}
\usepackage{epsfig}
\usepackage{amsmath,amssymb}
\usepackage{enumitem}
\usepackage{multirow}
\usepackage[latin1]{inputenc}
\usepackage{epic,eepic}
\usepackage{overpic}
\usepackage{color}
\usepackage{amsfonts}
\usepackage{graphicx}
\usepackage{enumitem}

\newenvironment{proof}{\smallskip\noindent{{\it Proof.}}\hskip \labelsep}%
            {\hfill\penalty10000\raisebox{-.09em}{\large\bf\rm $\blacksquare$}\par\medskip}

\newtheorem{theorem}{Theorem}[section]
\newtheorem{lemma}[theorem]{Lemma}
\newtheorem{proposition}[theorem]{Proposition}
\newtheorem{corollary}[theorem]{Corollary}

\newtheorem{remark}{Remark}[section]

\def\bR{\mathbf{R}}
\def\bO{\mathbf{O}}
\def\bS{\mathbf{S}}

\definecolor{verd}{rgb}{0.1, 0.5, 0.1}

\begin{document}

\begin{frontmatter}

\title{Monotone cubic spline interpolation for functions with a strong gradient}\tnotetext[label1]{This research has been supported by Spanish MINECO projects MTM2017-83942-P.}

\author[UV1]{Francesc Ar\`andiga}
\ead{arandiga@uv.es}
\author[UV1]{Antonio Baeza}
\ead{antonio.baeza@uv.es}
\author[UV1]{Dionisio F. Y\'a\~nez}
\ead{dionisio.yanez@uv.es}
\date{Received: date / Accepted: date}

\address[UV1]{Departament de Matem\`atiques.  Universitat de Val\`encia (EG) (Spain)}


%
%

\begin{abstract}

Spline interpolation has been used in several applications
due to its favorable properties regarding smoothness and accuracy of the interpolant. However, when there exists a discontinuity or a steep gradient in the data, some
artifacts can appear due to the Gibbs phenomenon. Also, preservation of data monotonicity is a requirement in some applications, and that property is not automatically verified by the interpolator. In this paper, we study sufficient conditions to obtain monotone cubic
splines based on Hermite cubic interpolators and propose different ways to construct them using non-linear formulas. The order of approximation, in each case, is calculated and several numerical experiments are performed to contrast the theoretical results.
\end{abstract}


\end{frontmatter}

\section{Introduction and review: Hermite cubic interpolation}

Approximation techniques are used in applications as design of curves, surfaces, robotics, creation of pieces in industry  and many others due to the fact that they present certain regularity  properties. In particular, Hermite interpolatory polynomials have been developed  to obtain  interpolants of class $C^1$ that have been
applied, for example, to the numerical solution of differential equations (see \cite{abushamabialecki,floater}). We consider the problem of piecewise cubic Hermite interpolation, that can be stated as follows:
let $x_{1}< x_{2} < \hdots< x_{n}$ be a partition of the interval
$[x_{1} ,x_{n}]$ and let $f_{i} = f(x_{i})$ be the values of a
certain function at the knots. Given approximate values of the first derivative of

 $f$ at the knots $\{x_{i}\}_{i=1}^n,$ denoted by $\{\dot{f}_{i}\}_{i=1}^n,$ construct a  piecewise cubic polynomial function $P(x),$ conformed by $n-1$
cubic polynomials $P_i(x)$  defined on the ranges $[x_i,x_{i+1}],$ that satisfy
\begin{equation}
\begin{array}{ll}
P(x_i)=f_i, & P(x_{i+1})=f_{i+1},\\
P{'}(x_i)=\dot{f}_i, & P{'}(x_{i+1})=\dot{f}_{i+1}.\\
\end{array}
\end{equation}

We will use the following notation: the undivided differences of a function $f$ are denoted by $\Delta f_{i} = f_{i+1} - f_{i}$ and
$m_{i} ={\Delta f_{i}}/{h_{i}}$ denotes its divided differences, where
 $h_{i} = x_{i+1} -x_{i}$ are the mesh spacings and $\hat{h}=\max_{i=1,\hdots,n-1} (h_i)$.
The $i$-th polynomial
$P_i(x),  x \in [x_{i},x_{i+1}]$ (see \cite{Boo01}
for details),
has the form
\begin{equation}\label{Hermite}
P_{i}(x) = c^i_{1} + c^i_{2} (x - x_{i}) + c^i_{3} (x-x_{i})^{2} +
c^i_{4}(x-x_{i})^{2}(x-x_{i+1}),
\end{equation}
where:
\begin{equation}\label{Hermitecoef}
\begin{split}
 c^i_{1} &:= f_{i}, \ c^i_{2} := f[x_{i},x_{i}] = \dot{f}_{i}, \
 c^i_{3} := f[x_{i},x_{i}, x_{i+1}]
= ( {m_{i} -\dot{f}_{i}} )/ {h_{i}} \notag, \\
 c^i_{4} &:= f[x_{i},x_{i},x_{i+1},x_{i+1}] =
( {\dot{f}_{i+1} + \dot{f}_{i} - 2 m_{i}} ) / {h_{i}^{2}}.
\end{split}
\end{equation}

Hence, a procedure to compute $\{\dot{f}_{i}\}$ defines an algorithm
for constructing a cubic Hermite interpolant.

%

In some problems, it is required that the interpolant employed preserves the monotonicity of the data.   This problem has been tackled in the literature (see, e.g., \cite{Hyman, WA02, Bica, cripshussain, arandigabaezayanez}) leading to several options for monotonic Hermite interpolation. In the remaining of this section we cite some  known results dealing with  conditions
for a cubic Hermite interpolants to be monotonicity preserving and about its accuracy.

%
\begin{theorem}(Necessary conditions {for monotonicity}.)
\label{FrCarl}
Let $P_{i}$ be a monotone cubic Hermite interpolant
of the data $\{(x_{i},f_{i}, \dot f_i), (x_{i+1},f_{i+1},\dot f_{i+1})\}$. Then:
\begin{equation}\label{signos}
sign(\dot{f}_{i}) = sign(\dot{f}_{i+1}) = sign(m_{i}).
\end{equation}
Furthermore, if $m_{i} = 0$ then $P_{i}$ is monotone
(constant)
if and only if  $\dot{f}_{i} = \dot{f}_{i+1} = 0 $.
\end{theorem}

%

\begin{theorem} (Sufficient conditions {for monotonicity}
  \cite{FrCar}.) \label{le:BoorSwa}
  Let $I_{i}=[x_{i},x_{i+1}]$ and $P_{i}$ be a cubic Hermite  interpolant
  of the
data $\{(x_{i},f_{i},\dot f_i), (x_{i+1},f_{i+1},\dot f_{i+1})\}$, and let
$
\alpha_{i} := {\dot{f}_{i}} / {m_{i}}$,  $\beta_{i} :=
{\dot{f}_{i+1}} / {m_{i}}$. 
If

\begin{equation}
\label{eq:BoorSwa}
 0\leq {\alpha}_{i}, {\beta}_{i} \leq 3,
\end{equation}
then the resulting cubic Hermite interpolant (\ref{Hermite}) is monotone on $I_i$.
\end{theorem}


The following theorem provides a global result (see \cite{Huy93}).  This is a generalization of Theorem \ref{le:BoorSwa}.
\begin{theorem} (\cite{Huy93}) \label{BoorSwa}
With the notation of Theorem \ref{le:BoorSwa},  if for all $i$, $2\leq i\leq n-1$,
\begin{equation}
\label{eq:BoorSwa2}
|\dot{f}_{i} |\leq 3 \min (|m_{i-1}|,|m_{i}|)
\end{equation}
then (\ref{Hermite}) is monotone in each
$[x_{i},x_{i+1}]$, $2\leq i \leq n-2$.
\end{theorem}

Finally, we show a more general theorem proved in \cite{FrCar}:

\begin{theorem} (Sufficient conditions {for monotonicity}) \label{le:BoorSwa2}
\label{th:BoorSwaFrBut}
  Let $I_{i}=[x_{i},x_{i+1}]$ and $P_{i}$ be a cubic Hermite  interpolant of the
data $\{(x_{i},f_{i},\dot f_i)$, $(x_{i+1},f_{i+1},\dot f_{i+1})\}$, and let
$
\alpha_{i} := {\dot{f}_{i}} / {m_{i}}$,  $\beta_{i} :=
{\dot{f}_{i+1}} / {m_{i}}$. 
  If one of the following conditions are satisfied
\begin{equation}
\label{eq:BoorSwaFrBut}
\begin{array}{l}
  0\leq \alpha_i +  \beta_i \leq 3, \vspace{0.3cm}
\\
\alpha_i^2+\alpha_i(\beta_i-6)+(\beta_i-3)^2 <0,
\end{array}
\end{equation}
$\forall i=1,\ldots,n-1$
then the resulting cubic Hermite interpolant (\ref{Hermite}) is monotone on $I_i$.
\end{theorem}

There exist many methods in the literature that deal with the problem of computing approximate derivative values in a way such that the resulting polynomials keep high  (ideally, maximal)  order of approximation and at the same time produce monotonicity-preserving reconstructions. They all face the problem that in order to ensure high order accuracy, the monotonicity--preserving property is lost and conversely. In \cite{FrCar} it is proved that the use of non-linear techniques is necessary to obtain third-order accurate interpolants (\cite{arandiga13}) using the following lemma:

\begin{lemma}  (\cite{Boo01})
\label{le:1}
Let us assume that $f(x)$ is smooth. If
$\dot{f}_{i}=f{'}(x_{i}) +O(h^{q_{i}})$ and
$\dot{f}_{i+1}=f{'}(x_{i+1}) +O(h^{q_{i+1}})$,
then the piecewise cubic Hermite interpolant (\ref{Hermite})
satisfies, on the interval $[x_i,x_{i+1}]$,
\begin{equation}
\label{eq:ordre}
P_i(x)=f(x)+O(h^{q}) \quad \mbox{where} \quad q=\min (4, q_i+1, q_{i+1}+1).
\end{equation}
\end{lemma}

In addition to  monotonicity preservation, in some applications some regularity is demanded. In order to obtain $C^2$ approximations we introduce cubic spline interpolation in Section \ref{splines}. Also, we determine sufficient conditions to obtain monotone spline cubic interpolants using the theorems presented above. In Section \ref{sec:derivatives2} we construct new monotone interpolants and study their properties regarding the order of approximation. Some numerical experiments are
shown in Section \ref{sec:num} in order to confirm the properties of the proposed algorithms. Finally, some conclusions are presented in Section \ref{sec:conclusion}

\section{Cubic spline interpolation}\label{splines}

In this section, we construct a cubic spline $P_i(x)$ verifying the  conditions
\begin{eqnarray}\label{splinesconditions1}
P_i^{(k)}(x_{i+1})&=&P_{i+1}^{(k)}(x_{i+1}), \qquad k=0,1,2, \\\label{splinesconditions2}
P_1'(x_1)&=&\dot{f}_1:=f'(x_1)=f'_1,\\\label{splinesconditions3}
P_{n-1}'(x_n)&=&\dot{f}_n:=f'(x_n)=f'_n,
\end{eqnarray}
with $i=1,\hdots,n-2$.

Our approach to construct the cubic spline is to start by Eq. \eqref{Hermite}, thus the conditions \eqref{splinesconditions2}--
\eqref{splinesconditions3} and \eqref{splinesconditions1} for the cases $k=0, 1$ are satisfied.
Conditions \eqref{splinesconditions1} for  $k=2$  will be used to obtain the appropriate approximations to the values of the first derivatives.


\subsection{Spline cubic interpolation from Hermite cubic form}\label{secHermite1}

As indicated above we start from Eq. \eqref{Hermite} and we impose that:
\begin{equation}
  \label{eq:cond}
P_i'' (x_{i+1})=P_{i+1}'' (x_{i+1}), \quad i=1,\hdots,n-2,
\end{equation}
thus, we have that:
\begin{equation*}
P_l''(x) = 2c_3^l +4c_4^l(x-x_{l})+2c_4^l(x-x_{l+1}), \quad l=i,i+1.
\end{equation*}
Then, using \eqref{Hermitecoef} we get
\begin{equation*}
\begin{split}
P_i''(x_{i+1}) &= 2c_3^i +4c_4^i h_i=2\left(\frac{m_{i} -\dot{f}_{i}}{h_{i}}\right)+4\left(\frac{\dot{f}_{i+1} + \dot{f}_{i} - 2 m_{i}}{h_{i}}\right)=2\frac{\dot{f}_{i}}{h_{i}}+4\frac{\dot{f}_{i+1}}{h_{i}}-6\frac{m_{i}}{h_i} ,\\
 P_{i+1}''(x_{i+1})& =  2c_3^{i+1} +2c_4^{i+1}(x_{i+1}-x_{i+2})=2\frac{m_{i+1}}{h_{i+1}} -2\frac{\dot{f}_{i+1}}{h_{i+1}}-2\left(\frac{\dot{f}_{i+2} + \dot{f}_{i+1} - 2 m_{i+1}}{h_{i+1}}\right)\\
& =-4\frac{\dot{f}_{i+1}}{h_{i+1}} -2\frac{\dot{f}_{i+2}}{h_{i+1}}+6\frac{m_{i+1}}{h_{i+1}}. \\
\end{split}
\end{equation*}
By Eq. \eqref{eq:cond}, we obtain
\begin{equation}
\begin{split}
&2\frac{\dot{f}_{i}}{h_{i}}+4\frac{\dot{f}_{i+1}}{h_{i}}-6\frac{m_{i}}{h_i} =-4\frac{\dot{f}_{i+1}}{h_{i+1}} -2\frac{\dot{f}_{i+2}}{h_{i+1}}+6\frac{m_{i+1}}{h_{i+1}} \Rightarrow \\
&\frac{\dot{f}_{i}}{h_{i}}+2\left(\frac{1}{h_{i}}+\frac{1}{h_{i+1}}\right)\dot{f}_{i+1} +\frac{\dot{f}_{i+2}}{h_{i+1}}=3\left(\frac{m_{i}}{h_i} +\frac{m_{i+1}}{h_{i+1}}\right)  \Rightarrow\\
&\frac{h_{i+1} }{h_{i}+h_{i+1}}  \, \dot{f}_{i}+
2  \, \dot{f}_{i+1}+
\frac{h_{i} }{h_{i}+h_{i+1}}  \, \dot{f}_{i+2}=3\left(
  \frac{h_{i+1}}{h_{i}+h_{i+1}}m_{i}  +\frac{h_{i}}{h_{i}+h_{i+1}}m_{i+1}\right). \\
\end{split}
\end{equation}
If we take the boundary conditions given in Eq. \eqref{splinesconditions2}--\eqref{splinesconditions3}, i. e. $\dot{f}_{1}=f'_1$ and $\dot{f}_{n}=f'_{n}$, we have the following system
\begin{equation}
\label{eq:splind}
\left\{
\begin{array}{lc}
2\, \dot{f}_2+ \mu_1\, \dot{f}_3=b_2, & \\
\lambda_i \, \dot{f}_{i}+ 2 \, \dot{f}_{i+1}+ \mu_i \, \dot{f}_{i+2}=b_{i+1},& i=2,\ldots,n-3, \\
\lambda_{n-2} \, \dot{f}_{n-2}+ 2 \, \dot{f}_{n-1}=b_{n-1},&  \\
\end{array}
\right.
\end{equation}
where $\lambda_i=\frac{h_{i+1}}{h_i+h_{i+1}}$,  $\mu_i=\frac{h_i}{h_i+h_{i+1}}$, $\lambda_i+\mu_i=1$, $i=1,\hdots,n-1$, and
\begin{eqnarray*}
b_2&=&3\, (\lambda_1 \, m_{1} +\mu_1 \, m_{2})- \lambda_1 \,f'_{1} , \\
b_{i}&=&3\, (\lambda_{i-1} \, m_{i-1} +\mu_{i-1} \, m_{i}),\quad i=3,\ldots, n-2, \\
b_{n-1}&=&3\, (\lambda_{n-2} \, m_{n-2} +\mu_{n-2} \, m_{n-1})
-  \mu_{n-2} \, f'_{n}.
\end{eqnarray*}
Thus, the system obtained is
\begin{equation}\label{primersistema}
A \dot F = B,
\end{equation}
with
\begin{equation}\label{splinesystem}
A=\left[\begin{array}{cccccc}
2&\mu_1&&&&\\
\lambda_2&2&\mu_2&&&\\
&\lambda_3&2&\mu_3&\\
&&\ddots&\ddots &\ddots&\\
&&&\lambda_{n-3}&2&\mu_{n-3}\\
&&&&\lambda_{n-2}&2\\
\end{array}
\right], \dot F=\left[\begin{array}{c}
\dot{f}_2\\
\dot{f}_3 \\
\dot{f}_4 \\
\vdots\\
\dot{f}_{n-3}\\
\dot{f}_{n-2}\\
\dot{f}_{n-1}
\end{array}
\right], B=\left[\begin{array}{c}
 3(\lambda_1 \, m_{1} +\mu_1 \, m_{2})-\lambda_1 \,f'_1\\
 3(\lambda_2 \, m_{2} +\mu_2 \, m_{3})\\
 3(\lambda_3 \, m_{3} +\mu_3 \, m_{4}) \\
\vdots
\\
 3(\lambda_{n-4} \, m_{n-4} +\mu_{n-4} \, m_{n-3})\\
 3(\lambda_{n-3} \, m_{n-3} +\mu_{n-3} \, m_{n-2})\\
 3(\lambda_{n-2} \, m_{n-2} +\mu_{n-2} \, m_{n-1})-\mu_{n-2} f'_n\\
\end{array}
\right].
\end{equation}
The equality $\lambda_i+\mu_i=1$ implies that the matrix $A$ is irreducibly diagonally dominant and hence non-singular. In order to calculate the order
of the approximate derivative values computed by solving \eqref{primersistema}, we prove the following lemma and theorem using the ideas presented in \cite{Stoerbulirsch} (Eq. 2.4.2.14).

\begin{lemma}(\cite{Stoerbulirsch})\label{lemma1}
Let be $0<m \in \mathbb{N}$, $0\leq \mu_i,\lambda_i\leq 1$ with $1\leq i\leq m$ such that
$$\lambda_i+\mu_i=1\quad i=1,\hdots,m,$$
and $A\in \mathbb{R}^{m\times m}$ defined as:
\begin{equation}\label{eq:matriuA}
A:=\left[\begin{array}{cccccc}
2&\mu_1&&&&\\
\lambda_2&2&\mu_2&&&\\
&\lambda_3&2&\mu_3&\\
&&\ddots&\ddots &\ddots&\\
&&&\lambda_{m-1}&2&\mu_{m-1}\\
&&&&\lambda_{m}&2\\
\end{array}
\right].
\end{equation}
Given $w\in \mathbb{R}^{m}$, if $z \in \mathbb{R}^{m}$ solves  $Az^T= w$
then
$$||z||_\infty\leq ||w||_\infty.$$
\end{lemma}
\begin{proof}
Let $i_0$ be such that $|z_{i_0}|=||z||_\infty$, then
\begin{equation}
\begin{split}
  ||w||_\infty & \geq |w_{i_0}| =|\lambda_{i_0+1}z_{i_0-1}+2z_{i_0}+\mu_{i_0+1}z_{i_0+1}|\\
               & \geq 2|z_{i_0}|-\lambda_{i_0+1}|z_{i_0-1}|-\mu_{i_0+1}|z_{i_0+1}|\\
               & \geq 2|z_{i_0}|-(\lambda_{i_0+1}+\mu_{i_0+1})|z_{i_0}|\\
               & = |z_{i_0}|= ||z||_\infty
  \end{split}
\end{equation}
\end{proof}


\begin{lemma}
\label{lema1sempre}
Let us assume that $f(x)\in C^{4}([x_1,x_n])$ and let $L>0$ be such that such that $|f^{(4)}(x)|\leq L$, for all $x\in[x_1,x_n]$. If there exists $K>0$ such that $\hat{h}/h_j \leq K$ for all $j=1,\hdots,n-1$ and
\begin{equation}\label{equationR}
R(i):=3\lambda_{i-1} \, m_{i-1} +3\mu_{i-1} \, m_{i}-\lambda_{i-1} \,f'_{i-1}-2 f'_{i} -\mu_{i-1} f'_{i+1}, \quad 2\leq i \leq n-1,
\end{equation}
with $\lambda_i,m_i,\mu_i$, $1\leq i\leq n-1$, previously defined, then:
\begin{equation}\label{Requation}
|R(i)| \leq  \left(\frac{17\,K+K^2}{16}+1\right)L \hat{h}^3= O(\hat{h}^3), \qquad 2\leq i \leq n-1.
\end{equation}
\end{lemma}
\begin{proof}
Let be $2\leq i \leq n-1$, using Taylor's expansions we have that there exist $\tau_j^i \in [x_1,x_n], j=1\hdots,5$ such that:
\begin{equation}\label{eq:deriva}
\begin{split}
m_{i-1}=&\frac{f(x_{i-1}+h_{i-1})-f(x_{i-1})}{h_{i-1}}=f'_{i-1}+\frac{h_{i-1}}{2}
   f''_{i-1}+\frac{h_{i-1}^2}{6} f^{(3)}_{i-1}+\frac{h_{i-1}^3}{24}
   f^{(4)}(\tau^i_1)\\
m_i=&\left(\frac{(h_{i-1}+h_i) }{h_i}f'_{i-1}+\frac{(h_{i-1}+h_{i})^2
   }{2 h_{i}}f''_{i-1}+\frac{(h_{i-1}+h_{i})^3}{6
   h_{i}}f^{(3)}_{i-1}+\frac{ (h_{i-1}+h_{i})^4}{24
   h_{i}}f^{(4)}(\tau^i_2)\right)\\
   &-\frac{h_{i-1}
   }{h_{i}}f'_{i-1}-\frac{h_{i-1}^2 }{2
   h_{i}}f''_{i-1}-\frac{ h_{i-1}^3}{6
   h_{i}}f^{(3)}_{i-1}-\frac{h_{i-1}^4}{24
   h_{i}}f^{(4)}(\tau^i_3)\\
=&
   f'_{i-1}+ \frac{(2 h_{i-1}+h_{i})}{2} f''_{i-1} +\frac{\left(3 h_{i-1}^2+3h_{i-1}h_{i}
   +h_{i-1}^2\right)}{6} f^{(3)}_1+\frac{ (h_{i-1}+h_{i})^4}{24
   h_{i}}f^{(4)}(\tau^i_2)-\frac{h_{i-1}^4}{24
   h_{i}}f^{(4)}(\tau^i_3)\\
f'_{i}=&f'_{i-1}+h_{i-1}f''_{i-1}+\frac{h_{i-1}^2}{2}f^{(3)}_{i-1}+\frac{h_{i-1}^3}{6}
   f^{(4)}(\tau^i_4)\\
f'_{i+1}=&f'_{i-1}+(h_{i-1}+h_{i})f''_{i-1}+\frac{(h_{i-1}+h_{i})^2}{2}f^{(3)}_{i-1}+\frac{(h_{i-1}+h_{i})^3}{6}
   f^{(4)}(\tau^i_5)\\
   \end{split}
\end{equation}
then by \eqref{eq:deriva} we have:
\begin{equation*}
\begin{split}
|R(i)|=&|3\lambda_{i-1} \, m_{i-1} +3\mu_{i-1} \, m_{i}-\lambda_{i-1} \,f'_{i-1}-2 f'_{i} -\mu_{i-1} f'_{i+1}|\\
=&|\lambda_{i-1}(3\,m_{i-1}-f'_{i-1})+\mu_{i-1}(3\,m_{i}-f'_{i+1})-2f'_{i}|\\
=&\bigg|\frac{1}{2(h_{i-1}+h_{i})}(\left(4 h_{i}
   f'_{i-1}+(3 h_{i-1}h_{i})f''_{i-1}+(h_{i-1}^2 h_{i})f^{(3)}_{i-1}\right)+\\
   &+\left(4h_{i-1}f'_{i-1}+\left(4 h_{i-1}^2
   +h_{i-1} h_{i} \right)f''_{i-1}+\left(2h_{i-1}^3+2h_{i-1}^2 h_{i}
   \right)f^{(3)}_{i-1}\right))\\
   &-2\left(f'_{i-1}+h_{i-1}f''_{i-1}+\frac{h_{i-1}^2}{2}f^{(3)}_{i-1}\right)+\frac{h_{i-1}^3h_{i}}{8(h_{i-1}+h_{i})}f^{(4)}(\tau^{i}_{1})+\frac{ h_{i-1}(h_{i-1}+h_{i})^3}{8
   h_{i}}f^{(4)}(\tau^i_2)\\&-\frac{h_{i-1}^5}{8
   h_{i}(h_{i-1}+h_{i})}f^{(4)}(\tau^i_3)- \frac{h_{i-1}(h_{i-1}+h_{i})^2}{6}
   f^{(4)}(\tau^i_5)-\frac{h_{i-1}^3}{3}
   f^{(4)}(\tau^i_4)\bigg|\\
\leq & \left(\frac{h_{i-1}^3h_{i}}{8(h_{i-1}+h_{i})}+\frac{ h_{i-1}(h_{i-1}+h_{i})^3}{8
   h_{i}}+\frac{h_{i-1}^5}{8
   h_{i}(h_{i-1}+h_{i})}+ \frac{h_{i-1}(h_{i-1}+h_{i})^2}{6}+\frac{h_{i-1}^3}{3}\right) L\\
   \leq & \left(\frac{K\hat{h}^3}{16}+K\hat{h}^3+K^2\frac{\hat{h}^3}{16}+ \frac{2\hat{h}^3}{3}+\frac{\hat{h}^3}{3}\right) L\\
 =& \left(\frac{17\,K+K^2}{16}+1\right)L \hat{h}^3= O(\hat{h}^3).
\end{split}
\end{equation*}

\end{proof}

\begin{theorem}  (\cite{Stoerbulirsch})
\label{teo:1}
Let us assume that $\hat h < 1$, $f(x)\in C^{4}([x_1,x_n])$ and  let $L>0$ be such that  $|f^{(4)}(x)|\leq L$, for all $x\in[x_1,x_n]$. $F'=[f'(x_2),\hdots,f'(x_{n-1})]^T$, $\dot F$ is the solution of system \eqref{primersistema} and there exist $K>0$, such that $\hat{h}/ h_j \leq K$ for all $j=1,\hdots,n-1$  then:
$$||\dot F-F'||_\infty \leq  O(\hat{h}^3). $$
\end{theorem}
\begin{proof}
We define $r=A(\dot F-F')=A\dot F - AF'=B-AF'$, by Lemma \ref{lema1sempre} we have:
\begin{equation*}
\begin{split}
|r_1|=&|b_2-2 f'_2 -\mu_1 f'_3|=|3\lambda_1 \, m_{1} +3\mu_1 \, m_{2}-\lambda_1 \,f'_1-2 f'_2 -\mu_1 f'_3|=|R(2)|\leq  \left(\frac{17\,K+K^2}{16}+1\right)L \hat{h}^3= O(\hat{h}^3),\\
|r_{i-1}|=&|b_i-\lambda_{i-1} \,f'_{i-1}-2 f'_{i} -\mu_{i-1} f'_{i+1}|=|3\lambda_{i-1} \, m_{i-1} +3\mu_{i-1} \, m_{i}-\lambda_{i-1} \,f'_{i-1}-2 f'_{i} -\mu_{i-1} f'_{i+1}|=|R(i)|\\
\leq & \left(\frac{17\,K+K^2}{16}+1\right)L \hat{h}^3= O(\hat{h}^3), \quad 3\leq i \leq n-2,\\
\end{split}
\end{equation*}
\begin{equation*}
\begin{split}
|r_{n-2}|=&|b_{n-1}-\lambda_{n-2} f'_{n-2}-2 f'_{n-1} |=|3\, (\lambda_{n-2} \, m_{n-2} +\mu_{n-2} \, m_{n-1})
-  \mu_{n-2} \, f'_{n}-\lambda_{n-2} f'_{n-2}-2 f'_{n-1}|\\
=&|R(n-1)|\leq  \left(\frac{17\,K+K^2}{16}+1\right)L \hat{h}^3= O(\hat{h}^3).
\end{split}
\end{equation*}
Finally, by Lemma \ref{lemma1}, we obtain:
$$||r||_\infty = ||A(\dot F-F')||_\infty \Rightarrow ||(\dot F-F')||_\infty\leq ||r||_\infty\leq O(\hat{h}^3).$$
\end{proof}

 In the case of equally-spaced grids \cite{kershaw3} the estimate in Theorem \ref{teo:1} can be improved to
\begin{equation}\label{eq:derivordre4}
||\dot F-F'||_\infty \leq  O(\hat{h}^4)
\end{equation}

\begin{corollary}
The cubic Hermite interpolant \eqref{Hermite} obtained using the approximations to the derivatives solving the system \eqref{primersistema} has order of accuracy $O(\hat{h}^4)$.
\end{corollary}
\begin{proof}
This is direct by Lemma \ref{le:1} and Theorem \ref{teo:1}.
\end{proof}

Now, we indicate the following conditions on the values $\dot f_i$, $i=2,\hdots,n-2$, in order to obtain a monotone interpolator in each interval $[x_i, x_{i+1}]$. We can prove the following result using  Theorem \ref{BoorSwa}.
\begin{theorem}
If for all $2\leq i \leq n-2$,
\begin{equation}
\left|\dot f_i\right|=\left|\sum_{j=1}^{n-2}A^{-1}_{ij}b_j\right|\leq 3\min (|m_{i-1}|,|m_{i}|)
\end{equation}
then the resulting cubic Hermite interpolant \eqref{Hermite} is monotone.
\end{theorem}
 These conditions, in general, are not satisfied.

If there exist any points  where the computed approximations of the derivative do not satisfy the sufficient conditions we will analyze two possibilities:
\begin{enumerate}
	 \item Change some of the derivative approximations, obtained by solving \eqref{primersistema}, by other values that produce a monotone interpolant. Hence, we define  for $i=2,\hdots, n-2$:
	\begin{equation}\label{primersplinenolineal}
	\begin{array}{l}
	\dot{f}_{i}^{M}:=\left \{
	\begin{array}{cl}
	\sum_{j=1}^{n-2}A^{-1}_{ij}b_j, &
	\hbox{if $ |\sum_{j=1}^{n-2}A^{-1}_{ij}b_j|\leq 3\min(|m_{i-1}|,|m_{i}|) $;} \\
	\tilde{f}_i, & \hbox{otherwise, }  \\
	\end{array}%
	\right.
	\end{array}
	\end{equation}
	where  $\tilde{f}_i$ is a  value calculated using non-linear techniques that will be explained in Section \ref{sec:derivatives2}.
	We only modify the approximations of the derivative values at the points where monotonicity constraints were not satisfied. In this way the interpolant keeps  the maximum order in every interval. With this method, the regularity is reduced to $C^1$ in a neighborhood of each point where the approximation of the derivative is modified. We analyze this case in Section \ref{monotonemaxorder}.

\item Another possibility is to try to keep $C^2$ regularity in the complete interval $[x_1, x_n]$ except at the single points where the approximations of the derivatives have been changed.
	Our proposal is to change the values where monitonicity was lost as in the previous case, rewrite  system \eqref{primersistema} but eliminating the modified points, and solve it. Afterwards, we again study if the new values satisfy the monotonicity conditions and repeat the process. We will prove in Section \ref{monotonemaxreg} that the order is lost in a neighborhood of the conflicting points, but is conserved at the rest.

\end{enumerate}

\subsection{Monotone spline with maximum order}\label{monotonemaxorder}

Assume that there exists a point $x_{i_0}$, with $1<i_0<n$, where the approximation to the derivative does not satisfy the conditions of the Theorem \ref{BoorSwa}. In that case, we change the value $\dot f_{i_0}$ by another value $\tilde f_{i_0}$. As a result the following equalities are not necessarily satisfied:
\begin{equation*}
\begin{split}
P''_{i_0-2}(x_{i_0-1})&=P''_{i_0-1}(x_{i_0-1}),\\
P''_{i_0-1}(x_{i_0})&=P''_{i_0}(x_{i_0}),\\
P''_{i_0}(x_{i_0+1})&=P''_{i_0+1}(x_{i_0+1}).
\end{split}
\end{equation*}
Thus, the regularity is $C^2$ in all points excepted at $x_{i_0+l}$ with $l=-1,0,1$. Finally, by Lemma \ref{le:1} the order is 4 except in the intervals $I_{i_0+l}$ with $l=-1,0,1$. We recapitulate these results in the following proposition.

\begin{proposition}\label{propo:maximoorden}
	Let us assume that $f(x)\in C^{4}([x_1,x_n])$ and let $L>0$ be such that $|f^{(4)}(x)|\leq L$, for all $x\in[x_{1},x_n]$. Let $p_{i_0},T>0$ and $A$, $\dot F$, $B$ be as defined in Eq. \eqref{splinesystem},
	which satisfy that $A\dot F=B$. If we define
	\begin{equation*}
	\dot F_{i_0}=\left\{
	\begin{array}{ll}
	\tilde{f}_{i_0}, &   i=i_0;\\
	\dot f_{i}, & i\neq i_0,
	\end{array}
	\right.
	\end{equation*}
	such that $|{f}'_{i_0}-\tilde{f}_{i_0}|=T\cdot \hat{h}^{p_{i_0}}$  and there exists $K>0$ such that $\hat{h}/h_i \leq K$ for all $i=1,\hdots,n$ then:
	\begin{equation}
	|\dot f_i-f'_i|=\left\{
	\begin{array}{ll}
	O(\hat{h}^{p_{i_0}}), &   i=i_0;\\
	O(\hat{h}^{3}), & i\neq i_0.
	\end{array}
	\right.
	\end{equation}
	Also, the cubic spline interpolator defined in Eq. \eqref{Hermite} using $F_{i_0}$ as an approximation of the values of the first derivatives, has $C^2$ regularity except at the points $x_{i_0+l}$, $l=-1,0,1$.
\end{proposition}

\subsection{Monotone spline with maximum regularity}\label{monotonemaxreg}

As a second option we replace the values for which the approximate first derivative does not satisfy the conditions in Theorem \ref{BoorSwa}, and recalculate the remaining values by rewriting  system \eqref{primersistema} in a way such that the equations corresponding to the changed values are removed from the system. Thus, we suppose that the approximation to the first derivative does not satisfy the sufficient conditions at $i_0$ with $1<i_0<n$. Consequently, we calculate $\dot{f}_{i_0}^{M}=\tilde{f}_{i_0}$ and define $A_{i_0^-} \in \mathbb{R}^{(i_0-2)\times (i_0-2)}$ and $A_{i_0^+} \in \mathbb{R}^{(n-1-i_0)\times (n-1-i_0)}$ by:
\begin{equation}\label{eq1subp}
A_{i_0^-}=\left[\begin{array}{cccccc}
2&\mu_1&&&&\\
\lambda_2&2&\mu_2&&&\\
&\lambda_3&2&\mu_3&\\
&&\ddots&\ddots &\ddots&\\
&&&\lambda_{i_0-3}&2&\mu_{i_0-3}\\
&&&&\lambda_{i_0-2}&2\\
\end{array}
\right],\,\, A_{i_0^+}=\left[\begin{array}{ccccccc}
2 & \mu_{i_0}&&&&&\\
\lambda_{i_0+1}&2&\mu_{i_0+1}&&&&\\
&\lambda_{i_0+2}&2&\mu_{i_0+2}&&&\\
&&&\ddots&\ddots &\ddots&\\
&&&&\lambda_{n-3}&2&\mu_{n-3}\\
&&&&&\lambda_{n-2}&2\\
\end{array}
\right],
\end{equation}
and the vectors:
\begin{equation}\label{eq2subp}
\quad \dot F_{i_0^-}=\left[\begin{array}{c}
\dot{f}_2\\
\dot{f}_3 \\
\dot{f}_4 \\
\vdots\\
\dot{f}_{i_0-3}\\
\dot{f}_{i_0-2}\\
\dot{f}_{i_0-1}
\end{array}
\right],\quad B_{i_0^-}=\left[\begin{array}{c}
 b_2\\
 b_3\\
 b_4 \\
\vdots
\\
b_{i_0-3}\\
b_{i_0-2}\\
b_{i_0-1}
\end{array}
\right] \quad \dot F_{i_0^+}=\left[\begin{array}{c}
\dot{f}_{i_0+1}\\
\dot{f}_{i_0+2} \\
\dot{f}_{i_0+3} \\
\vdots\\
\dot{f}_{n-3}\\
\dot{f}_{n-2}\\
\dot{f}_{n-1}
\end{array}
\right],\quad B_{i_0^+}=\left[\begin{array}{c}
 b_{i_0+1}\\
 b_{i_0+2}\\
 b_{i_0+3} \\
\vdots
\\
b_{n-3}\\
b_{n-2}\\
b_{n-1}
\end{array}
\right],
\end{equation}
\text{being:}
\begin{equation}\label{eq3subp}
\begin{split}
&b_2=3(\lambda_1 \, m_{1} +\mu_1 \, m_{2})-\lambda_1 \,f'_1,\\
&b_{n-1}=3(\lambda_{n-2} \, m_{n-2} +\mu_{n-2} \, m_{n-1})-\mu_{n-2} \, f'_n,\\
&b_{i_0-1}=3(\lambda_{i_0-2} \, m_{i_0-2} +\mu_{i_0-2} \, m_{i_0-1})-\mu_{i_0-2}\tilde{f}_{i_0},\\
&b_{i_0+1}=3(\lambda_{i_0} \, m_{i_0} +\mu_{i_0} \, m_{i_0+1})-\lambda_{i_0}\tilde{f}_{i_0},\\
&b_{i-1}=3(\lambda_{i-1} \, m_{i-1} +\mu_{i-1} \, m_{i}), \,\, \text{for} \,\,2<i<n-1,\,\,i\neq i_0+l,\,\,l=-1,0,1.\\
\end{split}
\end{equation}
With these variables we can rewrite the new system as $A_{i_0}\dot F_{i_0}=B_{i_0}$ where:
\begin{equation}\label{splinesystem21}
A_{i_0}=\left[\begin{array}{cc}
A_{i_0^-}& \\
&A_{i_0^+}\\
\end{array}
\right], \dot F_{i_0}=\left[\begin{array}{c}
\dot F_{i_0^-}\\
\dot F_{i_0^+}
\end{array}
\right], B_{i_0}=\left[\begin{array}{c}
B_{i_0^-}\\
 B_{i_0^+}
\end{array}
\right].
\end{equation}

We will prove that there exists a set of intervals around of $x_{i_0}$ where the order of the approximation to the first derivative is affected because of the modification of
 $\dot f_{i_0}$ but is maintained at the points that are sufficiently separated from the discontinuity.

  We adapt the results obtained in \cite{kershaw1,kershaw2,kershaw3}. For this, we divide the system $A_{i_0}\dot F_{i_0}=B_{i_0}$ in two subsystems and analyze them separately. Each subsystem is similar to the system obtained to construct a spline with different boundary conditions.

The following result \cite{kershaw1} provides a bound of the elements of the inverse.
\begin{lemma}\label{lemma2}
Let be $0<m \in \mathbb{N}$, $0\leq \mu_i,\lambda_i\leq 1$ with $1\leq i\leq m$ such that
$$\lambda_i+\mu_i=1\quad i=1,\hdots,m,$$
and $A\in \mathbb{R}^{m}\times \mathbb{R}^{m}$ defined
 by \ref{eq:matriuA}. If the elements of $A^{-1}$ are denoted by $A^{-1}_{ij}$ then $$|A_{ij}^{-1}|\leq \frac{2}{3}\cdot2^{-|i-j|},\quad 1\leq i,j\leq m. $$
\end{lemma}

\begin{remark}
In the case of uniform grid, i.e., when $h_i=h_{i-1}$, for all $i=2,\hdots,n$, then the bound can be improved. In \cite{kershaw1}, it is proved that:
\begin{equation}\label{cotanoespaciada}
|A_{ij}^{-1}|\leq \frac{2}{3}\cdot(2+\sqrt{3})^{-|i-j|},\quad 1\leq i,j\leq m.
\end{equation}
\end{remark}

\begin{proposition}\label{propo1}
Let us assume that $\hat{h}<1$, $f(x)\in C^{4}([x_1,x_{i_0}])$ and let $L>0$ be such that $|f^{(4)}(x)|\leq L$, for all $x\in[x_{1},x_{i_0-1}]$. Let  $p_{i_0},T>0$ and $A_{i_0^-}$, $\dot F_{i_0^-}$, $B_{i_0^-}$ be as defined in Eqs. \eqref{eq1subp}, \eqref{eq2subp} and \eqref{eq3subp}, which satisfy $|{f}'_{i_0}-\tilde{f}_{i_0}|=T\cdot \hat{h}^{p_{i_0}}$ and $A_{i_0^-}\dot F_{i_0^-}=B_{i_0^-}$. If $i_0>3-\log_2(\hat{h})$ and there exists $K>0$, such that $\hat{h}/h_i \leq K$ for all $i=1,\hdots,i_0-1$ then:

and  $\hat{h}<1$

There exists an integer $l_0<i_0-1$, such that:
\begin{equation}
|\dot f_i-f'_i|=\left\{
  \begin{array}{ll}
    O(\hat{h}^{\min(3,p_{i_0}+1)}), & 2\leq i \leq l_0; \\
    O(\hat{h}^{p_{i_0}}), &   l_0< i \leq i_0-1.
  \end{array}
\right.
\end{equation}
\end{proposition}

\begin{proof}
We define $r_{i_0^-}=((r_{i_0^-})_2, \hdots, (r_{i_0^-})_{i_0-1})^T$ as:

$$r_{i_0^-}=A_{i_0^-}(\dot F_{i_0^-}-F'_{i_0^-})=A_{i_0^-}\dot F_{i_0^-} - AF_{i_0^-}'=B_{i_0^-}-AF_{i_0^-}'.$$

Let $i$ be such that $2< i \leq i_0-2 $. By Lemma \ref{lema1sempre}:
\begin{equation}\label{eqteo1}
\begin{split}
|(r_{i_0^-})_i|=&|b_i-\lambda_{i-1} \,f'_{i-1}-2 f'_i-\mu_{i-1} f'_{i+1}|
=|3\lambda_{i-1} \, m_{i-1} +3\mu_{i-1} \, m_{i}-\lambda_{i-1} \,f'_{i-1}-2 f'_{i} -\mu_{i-1} f'_{i+1}|\\
=&|R(i)| \leq  \left(\frac{17\,K+K^2}{16}+1\right)L \hat{h}^3= O(\hat{h}^3).
\end{split}
\end{equation}
The result is proved analogously  for $i=2$. In the case $i=i_0-1$ we take
\begin{equation}\label{eqteo2}
\begin{split}
|(r_{i_0^-})_{i_0-1}|=&|b_{i_0-1}-\lambda_{i_0-2} \,f'_{i_0-2}-2 f'_{i_0-1}|\\
=&|b_{i_0-1}-\lambda_{i_0-2} \,f'_{i_0-2}-2 f'_{i_0-1}+(-\mu_{i_0-2}f'_{i_0}+\mu_{i_0-2}f'_{i_0})|\\
=& \left|\left(3(\lambda_{i_0-2} \, m_{i_0-2} +\mu_{i_0-2} \, m_{i_0-1})-\lambda_{i_0-2} \,f'_{i_0-2}-2 f'_{i_0-1}-\mu_{i_0-2}f'_{i_0}\right)+\mu_{i_0-2}(f'_{i_0}-\tilde{f}_{i_0})\right|\\
   \leq & |R(i_0-1)| +\mu_{i_0-2}\left|f'_{i_0}-\tilde{f}_{i_0}\right| \\
\leq & \left(\frac{17\,K+K^2}{16}+1\right)L \hat{h}^3+ \mu_{i_0-2}\cdot T\cdot h^{p_i}= O(\hat{h}^3)+\mu_{i_0-2}\cdot T\cdot h^{p_i}.
\end{split}
\end{equation}
Now, from
$$2^{i-(i_0-2)} \leq \hat{h} \leftrightarrow (i-(i_0-2))\log(2)\leq\log(\hat{h})\leftrightarrow i\leq(i_0-2)+\frac{\log(\hat{h})}{\log(2)}=(i_0-2)+\log_2(\hat{h})=:\tilde l_0$$
Also, we impose that
\begin{equation*}
1<\tilde l_0=(i_0-2)+{\log_2(\hat{h})} \rightarrow 3-{\log_2(\hat{h})}<i_0.
\end{equation*}
Then, by Lemma \ref{lemma2} we have for $1\leq i \leq \tilde l_0$:
\begin{equation*}
|A^{-1}_{i\,i_0-2}|\leq \frac23 2^{i-(i_0-2)} \leq \hat{h}
\end{equation*}
and by Eqs. \eqref{eqteo1} and \eqref{eqteo2}  if $1\leq i \leq \tilde l_0$
\begin{equation*}
|\dot f_{i+1}-f'_{i+1}|= \left|\sum_{j=1}^{i_0-2}A^{-1}_{i,j}r_{j+1}\right|\leq O(\hat{h}^3)+|A^{-1}_{i\,i_0-2}|\cdot T\cdot \hat{h}^{p_i}\leq O(\hat{h}^3)+\hat{h} \cdot O(\hat{h}^{p_i})=O(\hat{h}^{\min{(3,p_i+1)}}).
\end{equation*}
Thus, we define $l_0=\tilde l_0+1$ and obtain:
\begin{equation}
|\dot f_i-f'_i|=\left\{
  \begin{array}{ll}
    O(\hat{h}^{\min(3,p_{i_0}+1)}), & 2\leq i \leq l_0; \\
    O(\hat{h}^{p_{i_0}}), &   l_0< i \leq i_0-1.
  \end{array}
\right.
\end{equation}

\end{proof}

\begin{proposition}\label{propo2}
Let us assume that $\hat{h}<1,$ $f(x)\in C^{4}([x_{i_0+1},x_n])$  and let $L>0$ be such that $|f^{(4)}(x)|\leq L$, for all $x\in[x_{i_0+1}, x_n].$ Let  $p_{i_0},T>0$ and $A_{i_0^+}$, $\dot F_{i_0^+}$, $B_{i_0^+}$ be as defined in Eqs. \eqref{eq1subp}, \eqref{eq2subp} and \eqref{eq3subp},
which satisfy  $|{f}'_{i_0}-\tilde{f}_{i_0}|=T\cdot \hat{h}^{p_{i_0}}$ and $A_{i_0^+}\dot F_{i_0^+}=B_{i_0^+}$. If $n>i_0+1-\log_2(\hat{h})$ and there exist $K>0$, such that $\hat{h}/h_i \leq K$ for all $i=i_0+1,\hdots,n$ then:

There exists an integer $i_0+1<l_1$, such that:
\begin{equation}
|\dot f_i-f'_i|=\left\{
  \begin{array}{ll}
    O(\hat{h}^{\min(3,p_{i_0}+1)}), & l_1\leq i \leq n-1; \\
    O(\hat{h}^{p_{i_0}}), &   i_0+1\leq  i < l_1.
  \end{array}
\right.
\end{equation}
\end{proposition}

The proof is similar to Prop. \ref{propo1}, the key is again to use  Lemma \ref{lemma2}.

The following corollary summarizes the order obtained in the approximation of the derivatives if the original system \eqref{splinesystem} is changed by \eqref{splinesystem21}. The order ot the approximation is reduced in a neighborhood of a point $x_{i_0}$ where monotonicity constraints do not hold when the original system is used, and, in exchange, $C^2$ regularity is maintained, except at the point $x_{i0}$ itself. the corollary is a direct consequence of Propositions. \ref{propo1} and \ref{propo2}.

\begin{corollary}\label{coro1}
Let us assume that $\hat h < 1$, $f(x)\in C^{4}([x_1,x_n])$ and let $L>0$ be such that $|f^{(4)}(x)|\leq L$, for all $x\in[x_{1},x_n]$. Let  $p_{i_0},T>0$ and $A_{i_0}$, $\dot F_{i_0}$, $B_{i_0}$ be as defined in Eqs. \eqref{eq1subp}, \eqref{eq2subp} and \eqref{eq3subp}
which satisfy  $|{f}'_{i_0}-\tilde{f}_{i_0}|=T\cdot \hat{h}^{p_{i_0}}$ and $A_{i_0}\dot F_{i_0}=B_{i_0}$. If $3-\log_2(\hat{h})<i_0<(n-1)+\log_2(\hat{h})$  and there exists $K>0$, such that $\hat{h}/h_i \leq K$ for all $i=1,\hdots,n$ then:

There exist integers $l_0<i_0<l_1$, such that:
\begin{equation}
|\dot f_i-f'_i|=\left\{
  \begin{array}{ll}
    O(\hat{h}^{\min(3,p_{i_0}+1)}), & 2\leq i \leq l_0; \\
    O(\hat{h}^{p_{i_0}}), &   l_0<  i < l_1;\\
    O(\hat{h}^{\min(3,p_{i_0}+1)}), & l_1\leq i \leq n-1.
  \end{array}
\right.
\end{equation}
Moreover the  cubic spline interpolator defined in Eq. \eqref{Hermite} using $F_{i_0}$ as an approximation of the values of the first derivatives, has $C^2$ regularity except at point $x_{i_0}$.
\end{corollary}

In the case of a piecewise $C^4$ function that has a smaller smoothness
 at an interval $[x_{i_0}, x_{i_0+1}]$, the following result holds \cite{kershaw3}:

\begin{proposition}\label{prop3}
 	Let us assume that $\hat{h}<1$, $f(x)\in C^{4}([x_1,x_{i_0}]\cup[x_{i_0+1},x_{n}])$ and let $L>0$ be such that $|f^{(4)}(x)|\leq L$, for all $x\in[x_1,x_{i_0}]\cup[x_{i_0+1},x_{n}]$.
	Let  $r>0$, $p_{i_0},p_{i_0+1}\geq 0$ and $A_{i_0^-}$, $\dot F_{i_0^-}$, $B_{i_0^-}$, $A_{(i_0+1)^+}$, $\dot F_{(i_0+1)^+}$, $B_{(i_0+1)^+}$
	be as defined in Eqs. \eqref{eq1subp}, \eqref{eq2subp} and \eqref{eq3subp} which satisfy  $|{f}'_{l}-\tilde{f}_{l}|=T\cdot \hat{h}^{p_{l}}, \,\, l=i_0,i_0+1$ and $A_{i_0}\dot F_{i_0}=B_{i_0}$. If $3-r\log_2(\hat{h})<i_0<(n-1)+r\log_2(\hat{h})$ and there exists $K>0$ such that $\hat{h}/h_i \leq K$ for all $i=1,\hdots,n$ then
	there exist integers $\tilde l_0<i_0,i_0+1<\tilde l_1$, such that:
	\begin{equation}
	|\dot f_i-f'_i|=\left\{
	\begin{array}{ll}
	O(\hat{h}^{\min(3,r+p_{i_0})}), & 2\leq i \leq \tilde l_0; \\
	O(\hat{h}^{\min(p_{i_0},p_{i_0+1})}, &   \tilde l_0<  i < \tilde l_1;\\
	O(\hat{h}^{\min(3,r+p_{i_0+1})}), & \tilde l_1\leq i \leq n-1,
	\end{array}
	\right.
	\end{equation}
	being:
	\begin{equation*}
	\tilde l_0=(i_0-1)+r\log_2(\hat{h}),\quad \tilde l_1=(i_0+2)-r\log_2(\hat{h}).
	\end{equation*}
\end{proposition}

\section{Non-linear computation of  derivatives}
\label{sec:derivatives2}

In Sections \ref{monotonemaxorder} and \ref{monotonemaxreg} we have shown two ways to construct monotonicity-preserving cubic splines by replacing some derivative values whenever necessary. In both cases, we have shown the importance of the order of accuracy of the approximate derivative values used as a replacement.
In this section we review  different ways to design these values and we study the respective order obtained by the interpolant.

Let us start by a formula designed by Fritsch and Butland \cite{FrBut}.
\begin{equation}
  \label{eq:frbut}
\dot{f}_{i}^{FB}:=\left \{
\begin{array}{ccc}
\frac{3  m_{i-1}   m_{i} }
{      m_{i-1}  + 2  m_{i} }, & \mbox{  if  } &\hbox{ $| m_{i} | \leq | m_{i-1} | $;} \\
\frac{3  m_{i-1}  m_{i} }{ m_{i} + 2  m_{i-1} }, &\mbox{  if  } & \hbox{ $| m_{i} | > | m_{i-1} | $;} \\
    0, & \mbox{  if  } & \hbox {$m_{i-1} m_{i} \leq 0$.} \\
\end{array}%
\right.
\end{equation}
In this case the value $\dot{f}^{FB}_{i}$ is defined in a way that
automatically satisfies (\ref{eq:BoorSwa2}). If $f$ is smooth and  $m_{i-1}m_i>0$  then
   \begin{equation}\label{fb:order}
 \dot{f}^{FB}_i=
f'(x_i)+O(\hat{h}).
\end{equation}
By Lemma \ref{le:1}, the cubic Hermite  interpolant (\ref{Hermite}) is at least second-order accurate (see \cite{arandiga13}).

A second possibility is Brodlie's formula \cite{FrBut}:
\begin{equation}
  \label{eq:brodlie}
\dot{f}_{i}^{B}:=
\left \{
\begin{array}{cc}
\dfrac{(wl_i +wr_i )m_{i-1}m_{i} }{ {wl_i}{ m_{i}}+ {wr_i  }{ m_{i-1}}  }, &  m_{i}
m_{i-1} \geq 0 ;
\\[0.4cm]
0 &  m_{i}
m_{i-1} < 0 ;
\end{array}%
\right.
\end{equation}
where $wl_i={h_{i-1}} + 2{h_{i}} $,
$wr_i=2{h_{i-1}} + {h_{i}}$. Formula   (\ref{eq:brodlie}) is implemented in the PCHIP program of Matlab
(\cite{Mol04}).

Using the properties of the weighted harmonic mean the following results are proved in  \cite{arandiga13}:
\begin{lemma}\label{lem:derb}
 Let us assume that $f$ is smooth and   $m_im_{i-1}>0$. Then,
If $h_{i-1}\neq h_i$ then
 \begin{equation}
\label{b:order}
\dot{f}^B_i= f'(x_i)+
\left\{ \begin{array}{lcl}
O(\hat{h}^2) & \mbox{if} &  h_{i-1}=h_i=\hat{h}, \\
O(\hat{h}) & \mbox{if} &  h_{i-1}\neq h_i.
\end{array}
\right.
\end{equation}
   \end{lemma}

\begin{theorem}
  \label{teo:brod}
 Let us assume that  $m_im_{i-1}>0$.
 If the derivatives $\dot f_i$ are computed using  Brodlie's formula
  (\ref{eq:brodlie}),  then the cubic Hermite interpolant $P_i(x)$
 constructed from (\ref{Hermite}) is
  monotone. Moreover, if $f$ is smooth then
 $P_i(x)$ is  third-order accurate when $h_{i-1}=h_i$ $\forall i$
and second-order accurate otherwise.
 \end{theorem}

Brodlie's formula produces a third order interpolant only in the case of using equally-spaced grids.
 In \cite{arandigayanez19} (see also \cite{arandiga13}) Ar\`andiga and Y\'a\~nez introduce a new  method to compute the approximated derivatives based on the weighted harmonic mean that achieve third order of accuracy for non-uniform grids and preserves monotonicity. The proposed formula is:
 \begin{equation}
  \label{eq:ay}
\dot{f}_{i}^{AY}:=\left\{
\begin{array}{ccc}
sign(m_i)\dfrac{(h_{i}+h_{i-1})^{1/p_i}\,|m_{i-1}| \, |m_{i}| }
{\left(h_{i-1}\, |m_{i-1}|^{p_i}+ h_i \, |m_{i}|^{p_i}\right)^{1/p_i}  }
&\mbox{  if  } &  m_{i}
m_{i-1} \geq 0 ;\\[0.4cm]
0 &\mbox{  if  } &  m_{i}
m_{i-1} < 0,
\end{array}
\right.
\end{equation}
where $p_i=\max(1,\frac{\log(w_i)}{\log(3)})$ and $w_i=2\max(h_{i-1},h_i)/\min(h_{i-1},h_i)$. With this formula,  the following proposition holds:
\begin{proposition}(\cite{arandigayanez19})
  \label{teo:arandigayanez}
 Let us assume that  $m_im_{i-1}>0$ and there exists $K>0$ such that $\hat{h}/h_j <K$, for all $j=1,\hdots,n$.
 If  we obtain the derivatives using  Ar\`andiga-Y\'a\~nez's formula,
  (\ref{eq:ay}),  then the cubic Hermite interpolant (\ref{Hermite}) is
  monotone. Moreover, if $f$ is smooth then
 $P_i(x)$ is  third-order accurate.
 \end{proposition}

Note that in the case of equally-spaced grids the formulas by Ar\`andiga and Y\'a\~nez  and Brodlie coincide.

\section{Numerical experiments}\label{sec:num}

In this section we  present some experiments to verify the theoretical results previously obtained. In particular, we will divide our experiments in two subsections: in \ref{sec:accuracy} we  study the order of approximation of the different reconstructions using smooth or piecewise smooth functions. We perform two experiments with both equally- and not equally-spaced grids. In the first case methods $B$ and $AY$ are the same.

 On the other hand, in \ref{sec:monotonicity} we check the monotonicity property in cases where the function is unknown and only nodal values are given.

In this section each method will be identified by an acronym, being:

\begin{description}
\item[$\bS$:] Cubic spline with boundary conditions $S'(a)=f'(a)$, $S'(b)=f'(b)$. If we do not know these boundary conditions we impose $S'(a)=m_1$, $S'(b)=m_{n-1}$.
\item[$\bO_k$:] Method explained in Section \ref{monotonemaxorder}. The approximations of the first derivative are computed using system \eqref{primersistema}--\eqref{splinesystem}  and those values which do not satisfy the conditions by Theorem \ref{le:BoorSwa2} are replaced by new approximations obtained through the methods explained in Section \ref{sec:derivatives2}.

\item[$\bR_k$:] Method explained in Section \ref{monotonemaxreg}. It is a cubic spline  with boundary conditions $S'(a)=f'(a)$, $S'(b)=f'(b)$ but constructed using system \eqref{primersistema}--\eqref{splinesystem}
	instead of \eqref{splinesystem21}  because an approximate derivative value is modified.
\end{description}
For methods {$\bO$} and {$\bR$} we introduce the subscript $k$ to indicate the approximation to the derivative used, thus $k=FB,\, B,$ or $AY$, Eqs. \eqref{eq:frbut},  \eqref{eq:brodlie} and \eqref{eq:ay} respectively.

\subsection{Accuracy}\label{sec:accuracy}
We divide this section in two parts: Firstly, we analyze the case of equally-spaced grids. We will check that the order of accuracy of the approximation to the derivatives' values is four at smooth parts. Secondly, we perform some experiments using a non-uniform grid to discretize the functions. In both cases, we explore the order of approximation at the points depending on the distance to the discontinuity.

\subsubsection{Experiments with uniform grids}\label{expsuniform}
In this section we consider two experiments: in the first one the function is smooth and we replace the approximation of the derivative at a single point to check the effect of this new value in the smoothness and accuracy of the spline; in the second one we consider a piecewise smooth function with a jump discontinuity.

\emph{Experiment 1.} In order to check the order of approximation of the methods we consider the following smooth function:
\begin{equation}\label{eq:exp1}
f(x)=x^4+\sin(x),
\end{equation}
and discretize it on $[0,2]$ using a uniform grid: $x_j^l=j/2^l$, $j=0,\hdots, 2^{l+1}$, being $l$ a fixed positive integer. We establish a window
$W\subseteq\{0, \dots, 2^{l+1}\},$ that selects a subset of the points in the discretization. Errors and numerical orders of the various methods are computed in the selected points, in order to verify the properties stated Sections  \ref{monotonemaxorder}, \ref{monotonemaxreg}  and  \ref{sec:derivatives2}. The errors are computed  in the window using:
$$e^W_l=\max_{i\in W}|f'(x_i)-\dot f_i|$$
and the order of accuracy of the approximation are estimated by computing
$$o^W=\log_2\left(\frac{e^W_l}{e^W_{l-1}}\right).$$

 For methods $\bO$ and $\bR$ we replace the derivative value  $\dot f_{2^l}$ corresponding to the point $i_0=2^l$ by new values computed by the methods in Section \ref{sec:derivatives2}, so as to verify the accuracy and smoothness properties stated in
	Sections \ref{monotonemaxorder} and \ref{monotonemaxreg}.

We first consider the window $W_1=\{i\,:\,0\leq i\leq 2^{l+1}\}$With this setup the order of accuracy is determined by the  approximation of the first derivative made in the point $i_0$. As shown in Table \ref{tab:a1}, in the case of $B$ and $AY$, we obtain second order in accordance with Eq. \eqref{b:order}; for $FB$ method, it is reduced by Eq. \eqref{fb:order}. Finally we remark that, according to \eqref{eq:derivordre4} the order is four for the $\bS$ algorithm as the grid is uniform.
\begin{table}[H]
  \begin{center}
   \begin{tabular}{|c|c|c|c|c|c|} \hline
   $h$   & $\bS$ &  $\bR_{FB}$  &  $\bO_{FB}$ & $\bR_{B}=\bR_{AY}$  &  $\bO_{B}=\bO_{AY}$   \\
   \hline
   $3.125e{-}2$  &$3.9988$  &    $ 0.9390$& $ 0.9390$&$ 1.9952$ &$ 1.9952$\\
   $1.562e{-}2$  &$3.9997$  &    $ 0.9715$& $ 0.9715$&$ 1.9988$ &$ 1.9988$\\
   $7.812e{-}3$  &$3.9999$  &    $ 0.9863$& $ 0.9863$&$ 1.9997$ &$ 1.9997$\\
   $3.906e{-}3$  &$3.9999$  &    $ 0.9932$& $ 0.9932$&$ 1.9999$ &$ 1.9999$\\
   \hline
   \end{tabular}
  \end{center}
   \caption{Experiment 1 with $h=2^{-l}$ (equally-spaced grid) and estimated orders
     $\log_2(e^{W_1}_l/e^{W_1}_{l-1})$,    $5\leq l \leq 8$,
$W_1=\{i\,:\,2\leq i\leq 2^{l+1}-1\}$.}
      \label{tab:a1}
\end{table}
If the window considered for order estimation is reduced so as to exclude $i_0$, that is, $W_2=W_1\setminus\{i_0\}$, the order for the $\bO$ methods is increased up to four, in agreement with Prop. \ref{propo:maximoorden}. In contrast, for the $\bR$ methods the order does not increase as the order reduction affects points close to $i_0,$ according to Cor. \ref{coro1} (see Table \ref{tab:a2}).
	
\begin{table}[H]
  \begin{center}
   \begin{tabular}{|c|c|c|c|c|c|} \hline
   $h$   & $\bS$ &  $\bR_{FB}$  &  $\bO_{FB}$ & $\bR_{B}=\bR_{AY}$  &  $\bO_{B}=\bO_{AY}$   \\
   \hline
   $3.125e{-}2$  &$3.9988$  &    $ 0.9390$& $3.9988$&$ 1.9952$ &$3.9988$\\
   $1.562e{-}2$  &$3.9997$  &    $ 0.9715$& $3.9997$&$ 1.9988$ &$3.9997$\\
   $7.812e{-}3$  &$3.9999$  &    $ 0.9863$& $3.9999$&$ 1.9997$ &$3.9999$\\
   $3.906e{-}3$  &$3.9999$  &    $ 0.9932$& $3.9999$&$ 1.9999$ &$3.9999$\\
   \hline
   \end{tabular}
  \end{center}
   \caption{Experiment 1 with $h=2^{-l}$  (equally-spaced grid)  and estimated orders
     $\log_2(e^{W_2}_l/e^{W_2}_{l-1})$,    $5\leq l \leq 8$,
$W_2=W_1\setminus\{i_0\}$.}
      \label{tab:a2}
\end{table}
Finally, if we consider the setup corresponding to Props. \ref{propo1} and \ref{propo2} by taking
\begin{equation*}
\begin{split}
l_0&=(i_0-1)+\log_2(\hat{h})=(2^l-1)+\log_2(2^{-l})=2^l-l-1,\\
l_1&=(i_0+1)-\log_2(\hat{h})=(2^l+1)-\log_2(2^{-l})=2^l+l+1,\\
\end{split}
\end{equation*}
and define $W_3=\{i\,:\,1\leq i\leq l_0 \,\,\text{or} \,\, l_1\leq i\leq 2^{l+1}-1\}$. As we can see in Table \ref{tab:a3}, the order increases in $\bR$ methods from two up to four in $B=AY$ methods and from one to three for $\bR_{FB}$. In this case, the size of the chosen window around the discontinuity  is sufficiently large as to increase the order of accuracy at the rest of points from the order obtained for all the points showed in Table \ref{tab:a2}. It would be possible to reduce this interval if we take the bound indicated in Eq. \eqref{cotanoespaciada} for equally-spaced grids. In this way, the constructed spline has maximum order and regularity at the interval $[x_1,x_{l_0}]\cup [x_{l_1},x_{n}]$ for all methods, according to Lemma \ref{le:1}.
 \begin{table}[H]
  \begin{center}
   \begin{tabular}{|c|c|c|c|c|c|} \hline
   $h$   & $\bS$ &  $\bR_{FB}$  &  $\bO_{FB}$ & $\bR_{B}=\bR_{AY}$  &  $\bO_{B}=\bO_{AY}$   \\
   \hline
   $3.125e{-}2$  &$3.9988$  &    $ 2.8397$& $3.9988$&$  3.8882$ &$3.9988$\\
   $1.562e{-}2$  &$3.9997$  &    $ 2.8713$& $3.9997$&$  3.9067$ &$3.9997$\\
   $7.812e{-}3$  &$3.9999$  &    $ 2.8864$& $3.9999$&$  3.8930$ &$3.9999$\\
   $3.906e{-}3$  &$3.9999$  &    $ 2.8932$& $3.9999$&$  3.9063$ &$3.9999$\\
   \hline
   \end{tabular}
  \end{center}
   \caption{Experiment 1 with $h=2^{-l}$ (equally-spaced grid) and estimated orders
     $\log_2(e^{W_3}_l/e^{W_3}_{l-1})$,    $5\leq l \leq 8$,
$W_3=\{i\,:\,1\leq i\leq l_0 \,\,\text{or} \,\, l_1\leq i\leq 2^{l+1}-1\}$.}
      \label{tab:a3}
\end{table}

\emph{Experiment 2.}   In this experiment we consider a piecewise smooth function with a jump discontinuity located in the interval $[x_{i_0}, x_{i_0+1}]$. The presence of the discontinuity produces two effects: first, the approximation of the derivatives obtained from \eqref{primersistema} will suffer the Gibbs phenomenon and produce some spurious oscillations near the discontinuity and hence monotonicity will not be preserved; second, the  methods discussed in Section \ref{sec:derivatives2} will not attain their maximum accuracy.

We consider the  function:
\begin{equation}\label{eq:exp2}
g(x)=\left\{
       \begin{array}{ll}
         x^4+\sin(x), & 0\leq x\leq 1, \\
         4+x^4+\cos(x), & 1 < x \le 2,
       \end{array}
     \right.
\end{equation}
and discretize it in $[0,2]$  analogously to previous subsection.This function has a jump discontinuity at $x=1$ that corresponds to the node $x_{2^l}.$

	 In Table \ref{tab:a4} we display the numerical orders obtained in the derivative approximation when we take $\tilde W_3=W_3\setminus\{2^l+l+1\}$ with $W_3$ being as defined in the previous experiment,
	 	 and in the left column of Fig. \ref{fig:exp2} the reconstructions produced by the different methods are shown. The spline $\bS$ produces oscillations due to the violation of monotonicity constraints and produces a poor order of accuracy in the reconstruction. 	 	
	 	
	 The same results are obtained for the $\bO$ methods because the aproximation of the first derivative is no modified in the nodes belonging to  $\tilde{W}_3$.
	 Finally there is an improvement in the order of accuracy if the $\bR$ methods are applied.
	 In the right plots of
	 	Fig. \ref{fig:exp2} the errors obtained in the derivative computation are shown. It can be seen that all methods produce big errors around the discontinuity, being the ones corresponding to the spline reconstruction one order of magnitude bigger than the rest. On the other hand, methods $B$ and $AY$ produce errors
	 	that are smaller by a factor of around $1/2$ than the ones produced by $FB$. Also, we observe that, as expected, the $\bO$ methods suffer an accuracy loss in a bigger neighbourhood around $[x_{i_0}, x_{i_0+1}]$ than the $\bR$ methods.

 \begin{table}[H]
  \begin{center}
   \begin{tabular}{|c|c|c|c|c|c|} \hline
   $h$   & $\bS$ &  $\bR_{FB}$  &  $\bO_{FB}$ & $\bR_{B}=\bR_{AY}$  &  $\bO_{B}=\bO_{AY}$   \\
   \hline
   $3.125e{-}2$  &$ 0.8964$  &     $1.8706$& $ 0.8964$&$ 2.0121$ &$ 0.8964$\\
   $1.562e{-}2$  &$ 0.8982$  &     $1.7824$& $ 0.8982$&$ 1.7898$ &$ 0.8982$\\
   $7.812e{-}3$  &$ 0.8991$  &     $1.8395$& $ 0.8991$&$ 1.8448$ &$ 0.8991$\\
   $3.906e{-}3$  &$ 0.8995$  &     $1.8693$& $ 0.8995$&$ 1.8723$ &$ 0.8995$\\
   \hline
   \end{tabular}
  \end{center}
   \caption{Experiment 2 with $h=2^{-l}$ and estimated orders
     $\log_2(e^{\tilde W_3}_l/e^{\tilde W_3}_{l-1})$,    $5\leq l \leq 8$,
$\tilde W_3=\{i\,:\,1\leq i\leq l_0 \,\,\text{or} \,\, l_1\leq i\leq 2^{l+1}-1\}$, when the grid is uniform.}
      \label{tab:a4}
\end{table}

\begin{figure}[!bpt]
\begin{center}
\begin{tabular}{ccc}
&(a) & (b)
 \\
\begin{rotate}{90}{$\qquad \qquad \quad \,\,\,\,\,\,\,\,\bS$}\end{rotate}
 &\includegraphics[width=6.0cm,height=4cm]{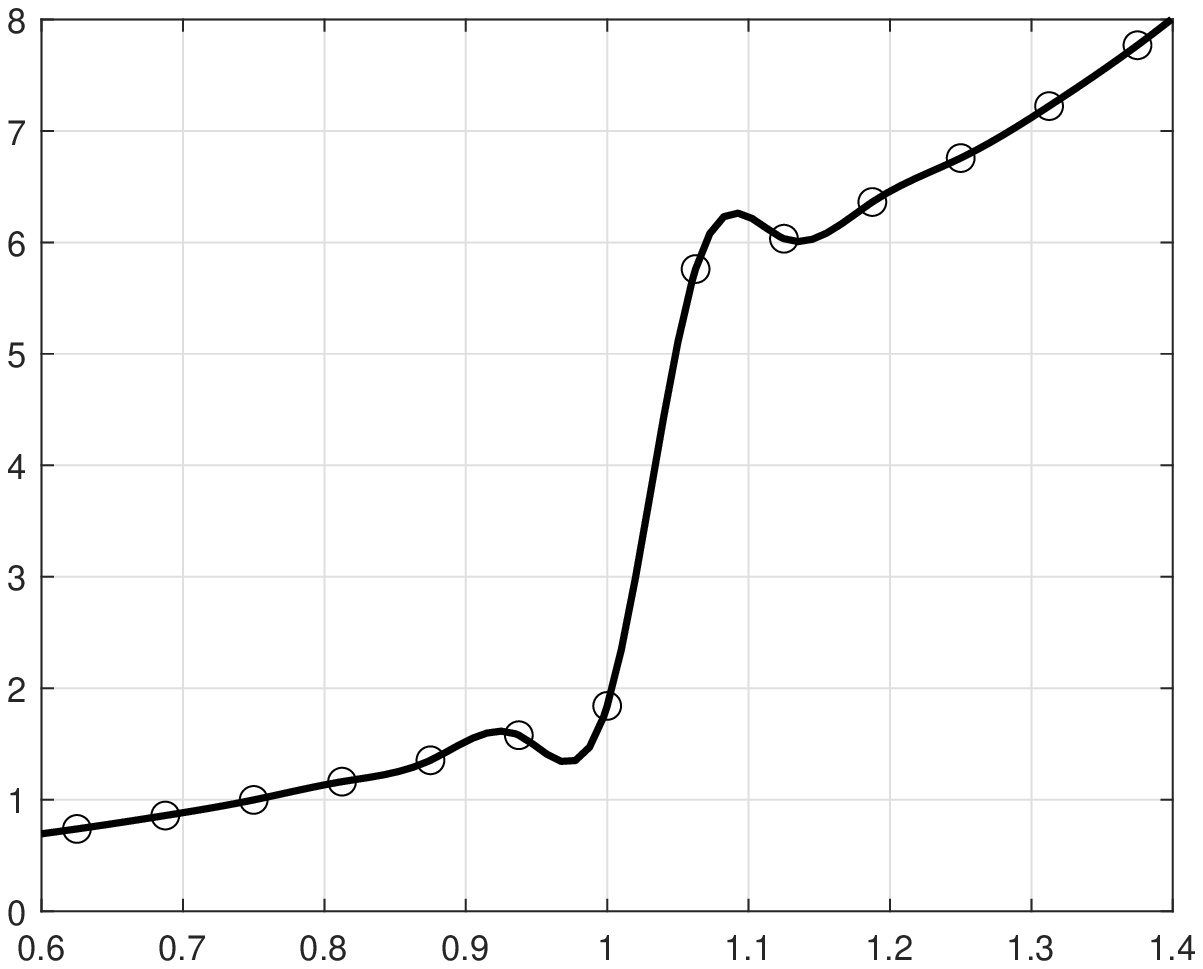} & \includegraphics[width=6.0cm,height=4cm]{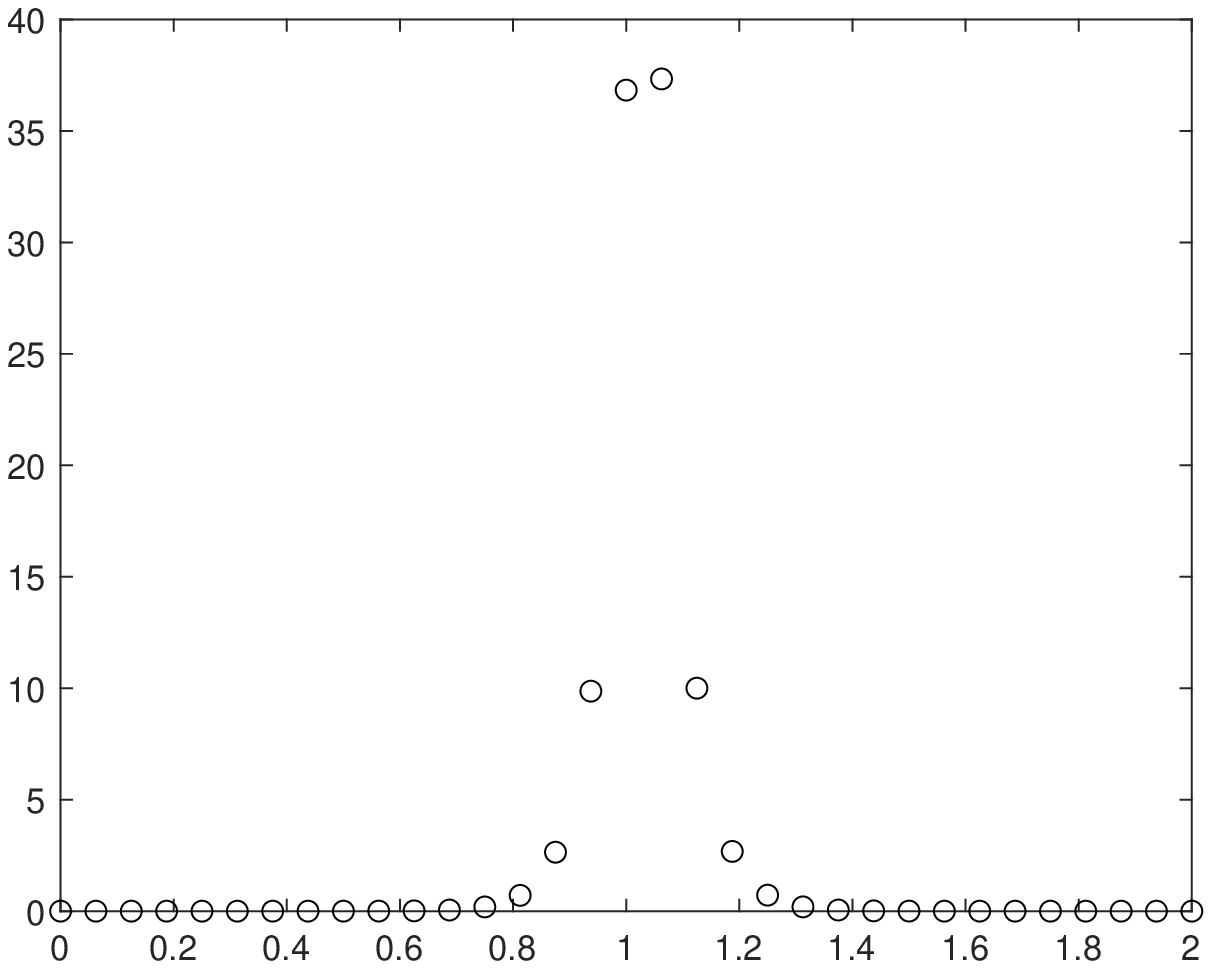}
 \\
\begin{rotate}{90}{$\qquad \qquad \quad \,\,\,\,\,\,\,\,\bR_{FB}$}\end{rotate} & \includegraphics[width=6.0cm,height=4cm]{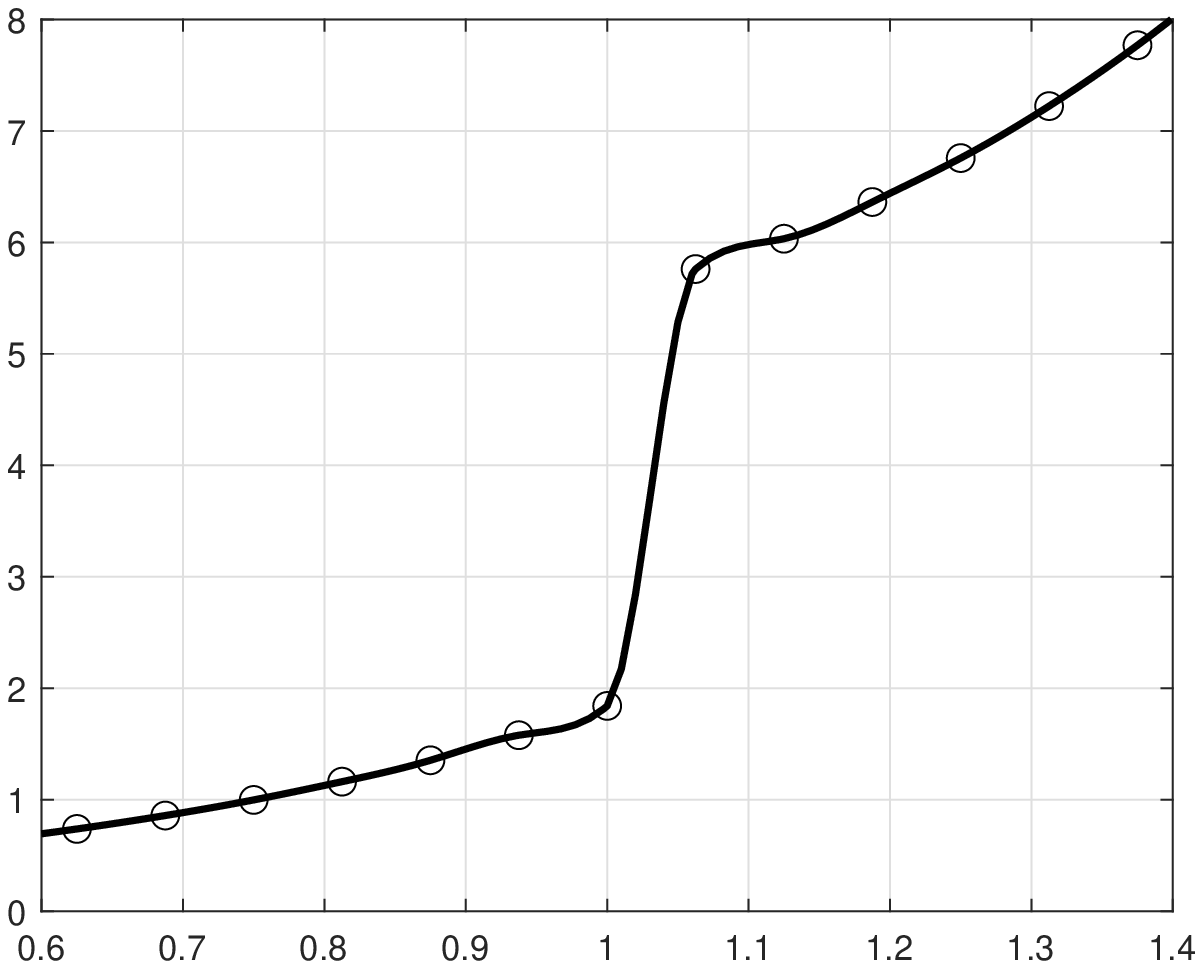}& \includegraphics[width=6.0cm,height=4cm]{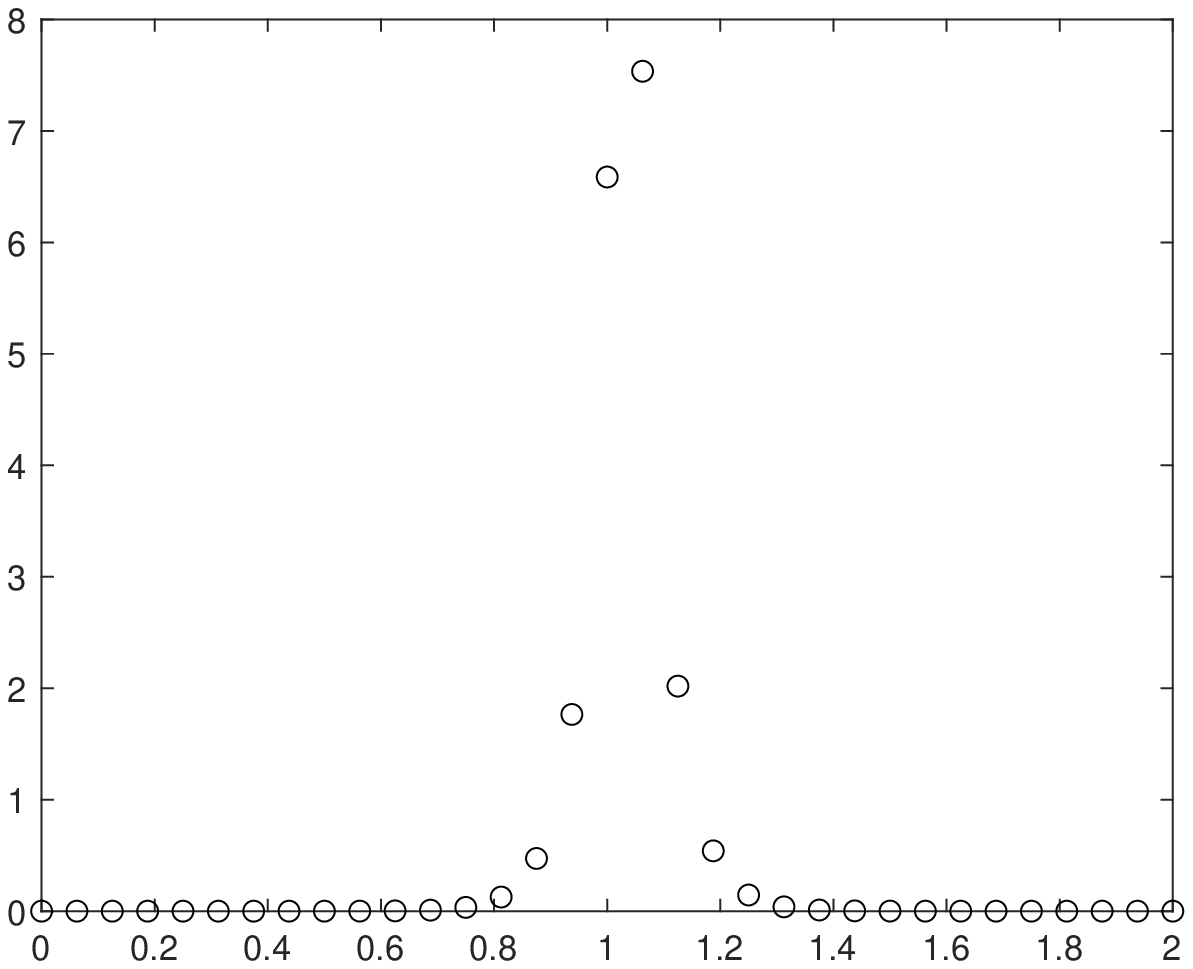}\\
\begin{rotate}{90}{$\qquad \qquad \quad \,\,\,\,\,\,\,\,\bO_{FB}$}\end{rotate} & \includegraphics[width=6.0cm,height=4cm]{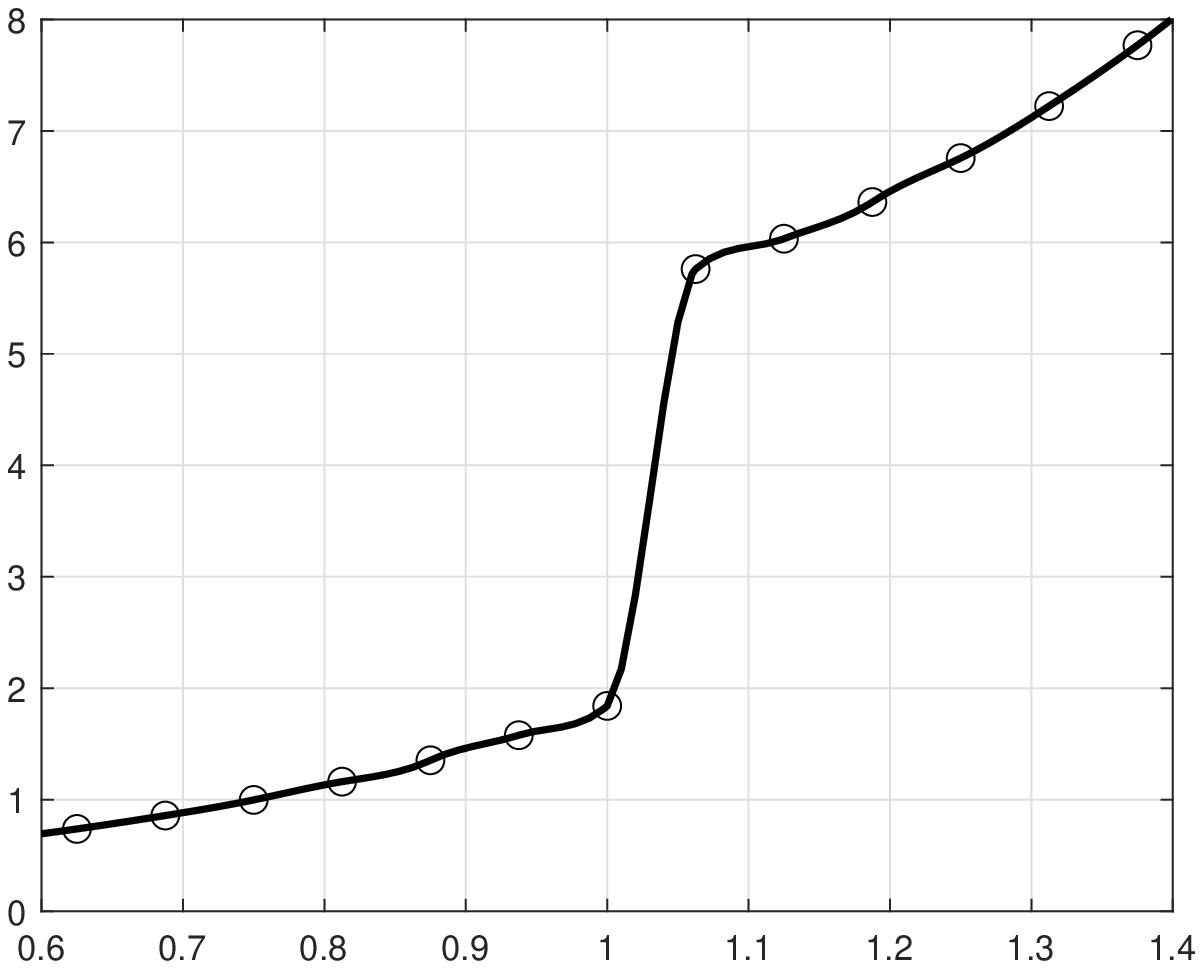} & \includegraphics[width=6.0cm,height=4cm]{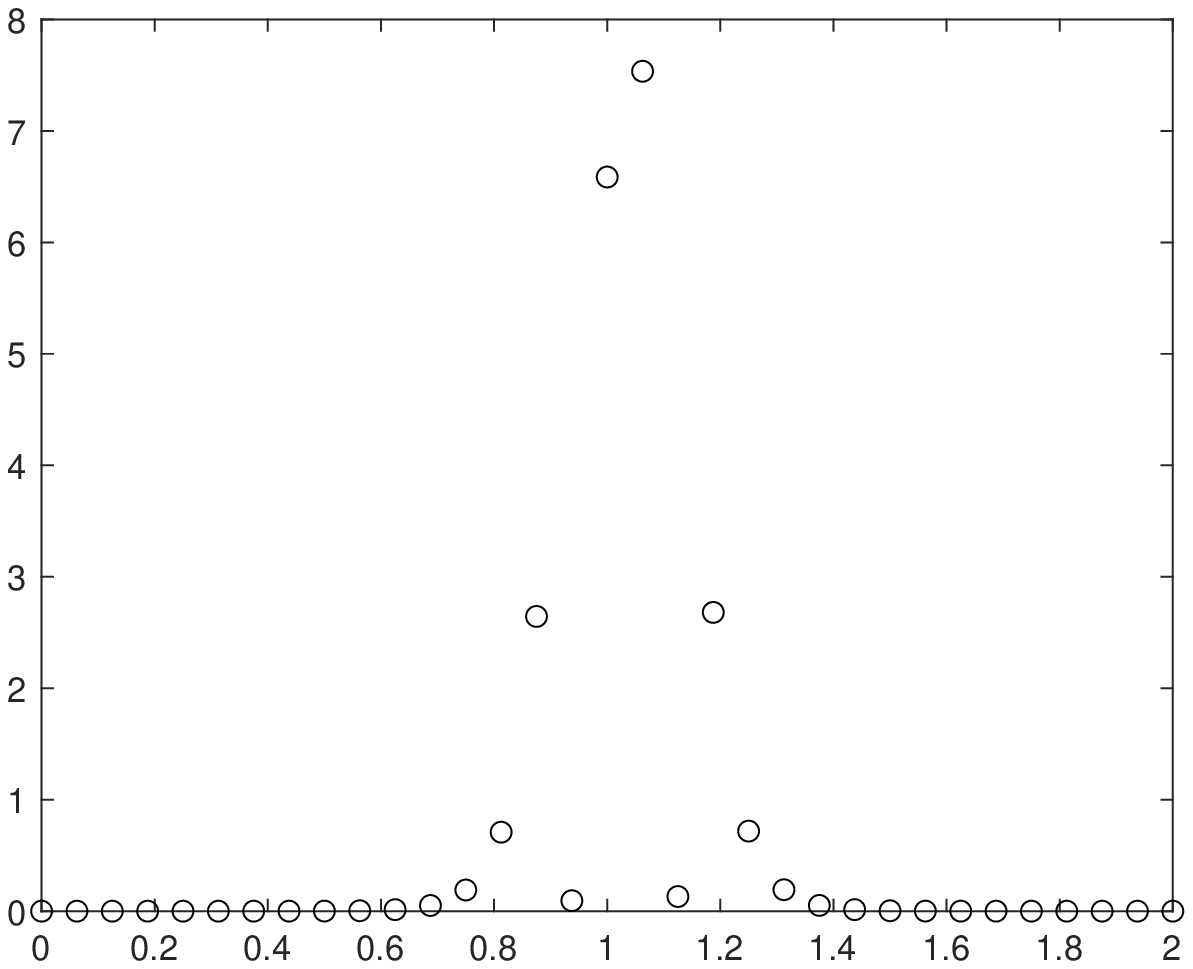}\\
\begin{rotate}{90}{$\quad \,\,\,\,\,\,\,\,\bR_{B}=\bR_{AY}$}\end{rotate}  & \includegraphics[width=6.0cm,height=4cm]{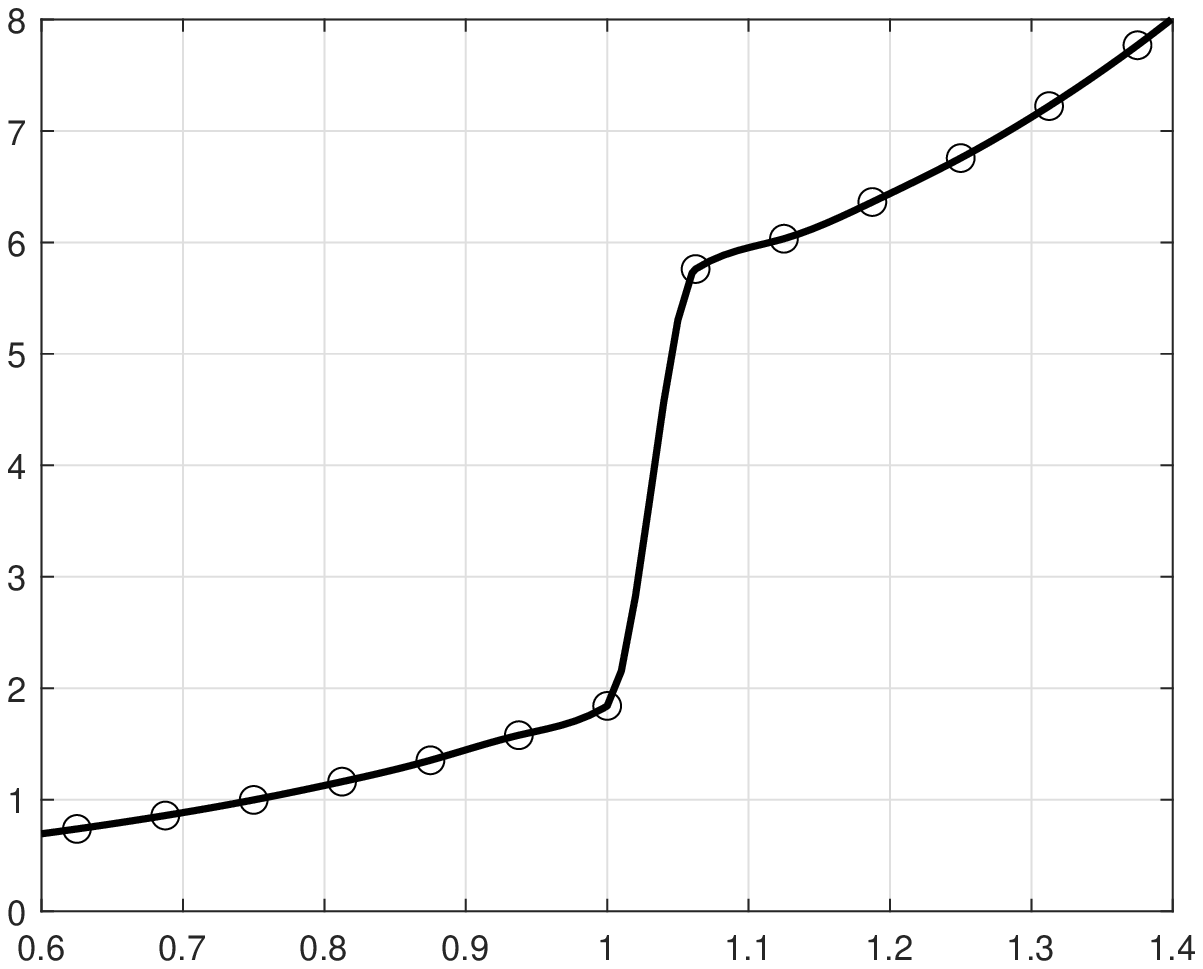}
&  \includegraphics[width=6.0cm,height=4cm]{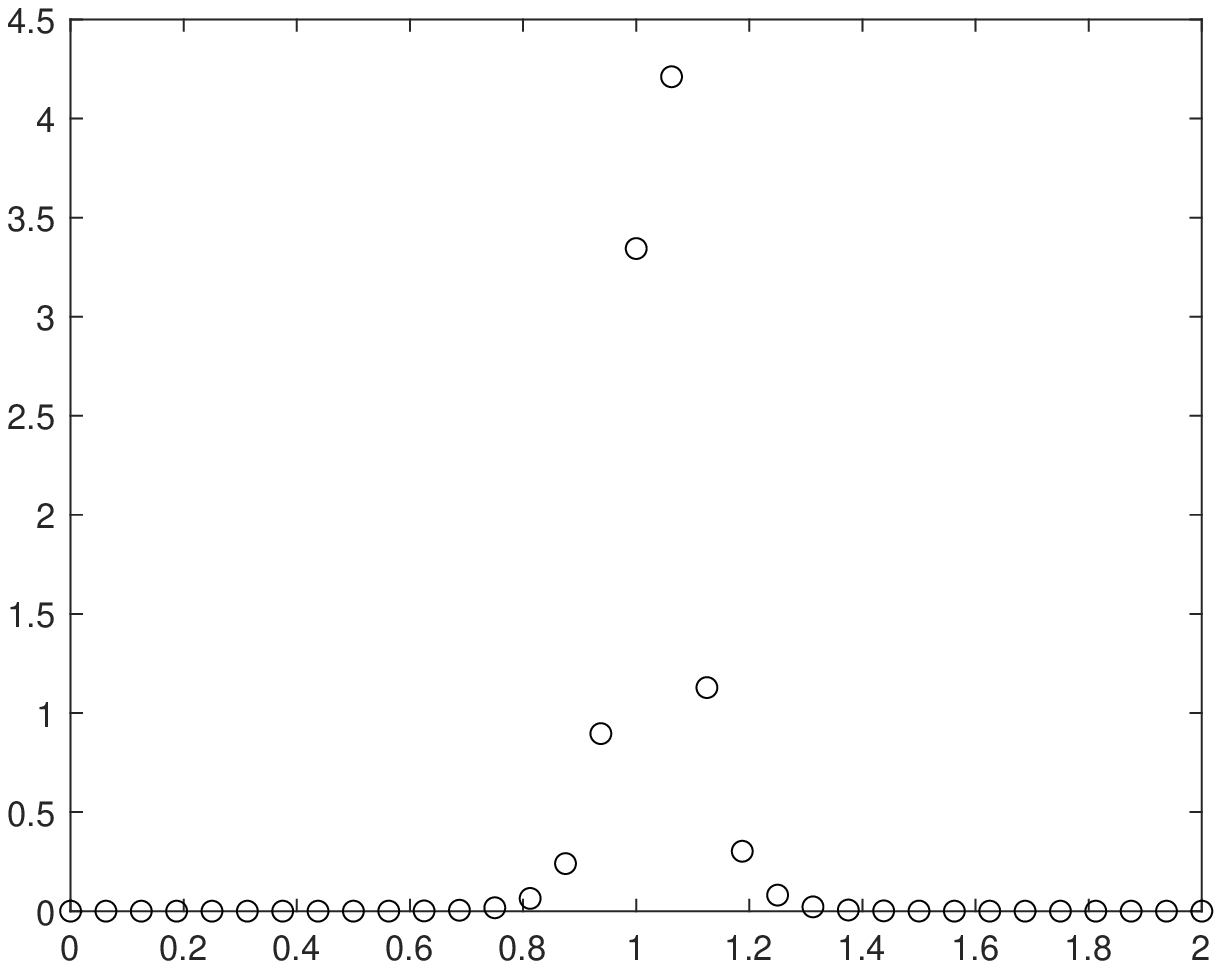}\\
\begin{rotate}{90}{$\quad \,\,\,\,\,\,\,\,\bO_{B}=\bO_{AY}$}\end{rotate}
& \includegraphics[width=6.0cm,height=4cm]{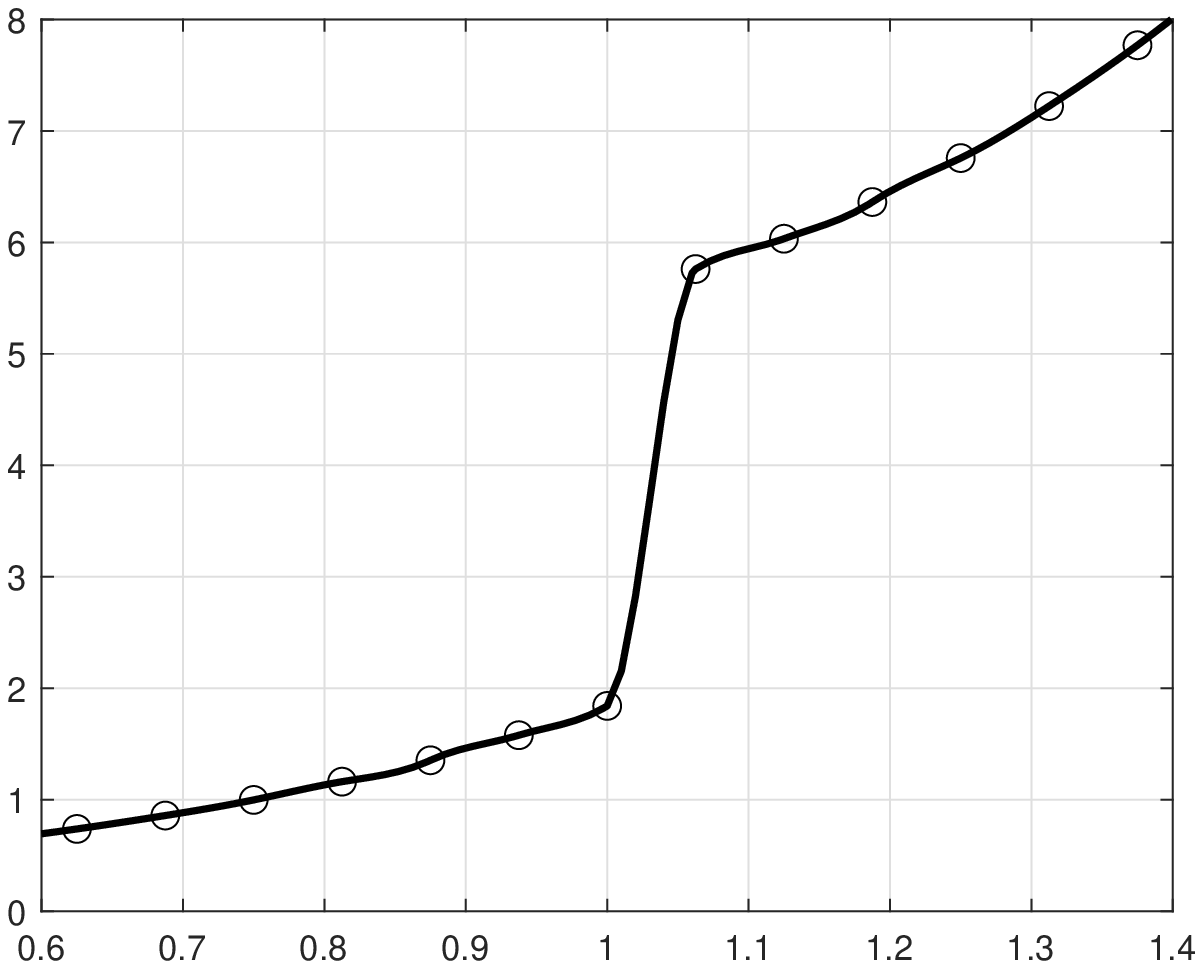} & \includegraphics[width=6.0cm,height=4cm]{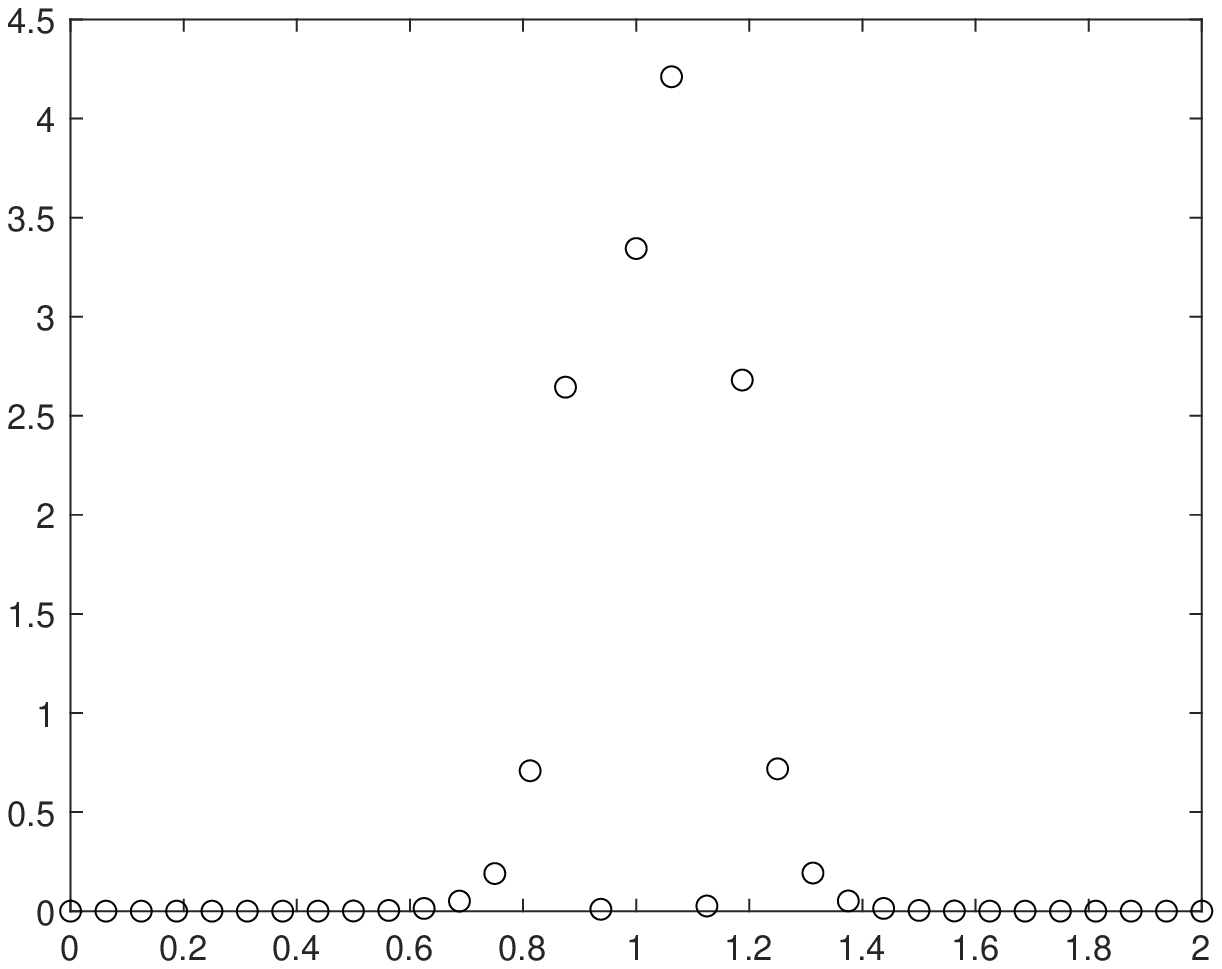} \\
\end{tabular}
\end{center}
\caption{Experiment 2. $l=4$ (a) Reconstructions using the different methods, (b) $|g'(x_i)-\dot g^k_i|$, $i=0,\hdots,32$.
  }
    \label{fig:exp2}
 \end{figure}

 As expected, the maximum order is not obtained. According to Prop. \ref{prop3}, we consider a window
	$W_4=\{i\,:\,1\leq i\leq \tilde l_0 \,\,\text{or} \,\, \tilde l_1\leq i\leq 2^{l+1}-1\}$ for
\begin{equation*}
\begin{split}
\tilde l_0&=(i_0-1)+r\log_2(\hat{h})=(2^l-1)+r\log_2(2^{-l})=2^l-r\cdot l-1,\\
\tilde l_1&=(i_0+2)-r\log_2(\hat{h})=(2^l+2)-r\log_2(2^{-l})=2^l+r\cdot l+2,\\
\end{split}
\end{equation*}
and $r=2$.  The results corresponding to this window are shown in Table \ref{tab:a5}. The accuracy orders for $\bO$ and $\bS$ coincide, whilst for $\bR$ the order is determined by the method used in the derivative computation.
 \begin{table}[H]
  \begin{center}
   \begin{tabular}{|c|c|c|c|c|c|} \hline
   $h$   & $\bS$ &  $\bR_{FB}$  &  $\bO_{FB}$ & $\bR_{B}=\bR_{AY}$  &  $\bO_{B}=\bO_{AY}$   \\
   \hline
   $3.125e{-}2$  &$2.7961$  &     $ 3.7667$& $2.7961$ &$ 3.9049$ &$2.7961$ \\
   $1.562e{-}2$  &$2.7980$  &     $ 3.6828$& $2.7980$ &$ 3.6906$ &$2.7980$ \\
   $7.812e{-}3$  &$2.7990$  &     $ 3.7397$& $2.7990$ &$ 3.7452$ &$2.7990$ \\
   $3.906e{-}3$  &$2.7995$  &     $ 3.7695$& $2.7995$ &$ 3.7727$ &$2.7995$ \\
   \hline
   \end{tabular}
  \end{center}
   \caption{Experiment 2 with $h=2^{-l}$ (equally-spaced grid) and estimated orders
     $\log_2(e^{W_4}_l/e^{W_4}_{l-1})$,    $5\leq l \leq 8$,
$W_4=\{i\,:\,1\leq i\leq \tilde l_0 \,\,\text{or} \,\, \tilde l_1\leq i\leq 2^{l+1}-1\}$.}
      \label{tab:a5}
\end{table}

\subsubsection{Experiments with  non-uniform grids}\label{expsnonuniform}
In order to check  the theoretical results in non-uniform grids, we consider the same functions as in subsection \ref{expsnonuniform} but using a different discretization.

\emph{Experiment 1 with a non-uniform grid} In this experiment we discretize the function $f(x)$ in \eqref{eq:exp1} at the interval $[0,2]$ using a non-uniform grid constructed according to the following procedure: Let $l$ be a fixed positive integer, we define the points as:
\begin{equation}\label{malladonouniforme}
\left\{
  \begin{array}{ll}
    x^l_{2i}=2^{-l}\cdot i, \\
    x^l_{2i+1}=2^{-l}(i+\frac 14),  \\
  \end{array}
\right.
\end{equation}
with $i=0,\hdots,2^{l+1}-1$ and $x^l_{2^{l+1}}=2$. It is clear that $\hat h=\frac{3}{4} 2^{-l}.$ As in the case of uniform grids, for $\bR$ and $\bO$ methods the value of the derivative computed by \ref{primersistema} at the point $\dot f_{2^{l}}$ is replaced by a new value computed using methods $FB$, $B$ and $AY$. In this case the node immediately at the right of the discontinuity, $x^l_{2^l+1}$, does not coincide with any node in the grid corresponding to $l+1$, but
	with the node $x^{l+2}_{2^{l+2}+2}$, belonging to the grid corresponding to $l+2$. Indeed:
$$
	x^l_{2^l+1} = 2^{-l} \left(2^{l-1}+\frac 14\right) =
	  2^{-l} \left(\frac{2^{l+1}+1}{4}\right)
	  =	  2^{-(l+2)} \left(2^{l+1}+1\right) =
	  x^{l+2}_{2(2^{l+1}+1)} = x^{l+2}_{2^{l+2}+2}.
	 $$	
	Taking it into account, in order to estimate the order of accuracy of the approximation we use the following formula:

$$\tilde o^W=\log_4\left(\frac{e^W_{l}}{e^W_{l-2}}\right).$$
We take the following windows, similar to the ones defined in Experiment 1 for uniform grids, i.e.:
\begin{equation*}
\begin{split}
&W^n_1=\{i\,:\,0\leq i\leq  2^{l+2}\},\\
&W^n_2=W^n_1\setminus\{i_0\},\\
&W^n_3=\{i\,:\,0\leq i\leq l^n_0 \,\,\text{or} \,\, l^n_1\leq i\leq 2^{l+2}\},
\end{split}
\end{equation*}
being
\begin{equation}\label{equationcotas}
\begin{split}
&l^n_0=(i_0-1)+\log_2(\hat{h})=(2^{l+1}-1)+\log_2\left(\frac{3}{4} 2^{-l}\right),\\
&l^n_1=(i_0+1)-\log_2(\hat{h})=(2^{l+1}+1)-\log_2\left(\frac{3}{4} 2^{-l}\right).\\
\end{split}
\end{equation}
It can be seen  in Tables \ref{tab:a6}, \ref{tab:a7} and \ref{tab:a8} that the orders of approximation obtained for $\bR_B$, $\bO_B$ and $\bR_{AY}$, $\bO_{AY}$ are different and are in accordance with Lemma \ref{teo:brod} and  Prop. \ref{teo:arandigayanez}. For the $\bS$ method the expected order, $O(\hat h^3)$ is obtained.

\begin{table}[H]
  \begin{center}
   \begin{tabular}{|c|c|c|c|c|c|c|c|} \hline
   $\hat h$   & $\bS$ &  $\bR_{FB}$  &  $\bO_{FB}$ & $\bR_{B}$ &$\bR_{AY}$   &  $\bO_{B}$ & $\bO_{AY}$    \\
   \hline
   $    9.3750e-02$  &$        2.9903$  & $      1.1996 $ & $    1.1996$&$     1.0888$ &$     1.9675$&$1.0888$&$      1.9675$\\
   $    2.3438e-02$  &$        2.9999$  & $      1.0791 $ & $    1.0791$&$     1.0385$ &$     2.0074$&$ 1.0385$&$     2.0074$\\
   $    5.8594e-03$  &$        3.0000$  & $      1.0220 $ & $    1.0220$&$     1.0107$ &$     2.0027$&$ 1.0107$&$     2.0027$\\
   $    1.4648e-03$  &$        2.9978$  & $      1.0057 $ & $    1.0057$&$     1.0027$ &$     2.0007$&$ 1.0027$&$     2.0007$\\
   \hline
   \end{tabular}
  \end{center}
   \caption{Experiment 1 with $\hat h=\frac 34 2^{-l}$ (non-equally-spaced grid) and estimated orders
     $\log_4(e^{W^n_1}_{l}/e^{W^n_1}_{l-2})$,    $3\leq l \leq 9$,
$W^n_1=\{i\,:\,0\leq i\leq 2^{l+2}\}$.}
      \label{tab:a6}
\end{table}

If the window is changed by $W^n_2$ so as to avoid the point where the derivative was replaced,
 the order increases for $\bO$ methods, in agreement with Prop. \ref{propo:maximoorden}, and these methods achieve the maximum order as they coincide with the $\bS$ method in the points of the window, see Table \ref{tab:a7}. However, the order is not improved using $\bR$ methods because the size of the window is not sufficiently large.
\begin{table}[H]
  \begin{center}
   \begin{tabular}{|c|c|c|c|c|c|c|c|} \hline
   $\hat h$   & $\bS$ &  $\bR_{FB}$  &  $\bO_{FB}$ & $\bR_{B}$ &$\bR_{AY}$  &  $\bO_{B}$ & $\bO_{AY}$    \\
   \hline
   $    9.3750e-02$  &$        2.9903$  & $        1.1623 $ & $   2.9903$&$         1.0653$ &$       1.9190$&$    2.9903$&$      2.9903$\\
   $    2.3438e-02$  &$        2.9999$  & $        1.0760 $ & $   2.9999$&$         1.0368$ &$       1.9964$&$    2.9999$&$     2.9999$\\
   $    5.8594e-03$  &$        3.0000$  & $        1.0218 $ & $   3.0000$&$         1.0106$ &$       2.0000$&$    3.0000$&$     3.0000$\\
   $    1.4648e-03$  &$        2.9978$  & $        1.0056 $ & $   2.9978$&$         1.0027$ &$       2.0000$&$    2.9978$&$     2.9978$\\
   \hline
   \end{tabular}
  \end{center}
   \caption{Experiment 1 with $\hat h=\frac342^{-l}$  (non-equally-spaced grid)  and estimated orders
     $\log_4(e^{W^n_2}_l/e^{W^n_2}_{l-2})$,    $3\leq l \leq 9$,
$W^n_2=W^n_1\setminus\{i_0\}$.}
      \label{tab:a7}
\end{table}

When the window
considered for the estimation of the order is $W^n_3$, i.e., some more  points around $i_0$ are excluded from the order estimation according to Props. \ref{propo1} and \ref{propo2}, the order of accuracy obtained in the  points in the window is optimal using any method. Even when the $FB$ method is used the order obtained is 3, instead of the expected second order (Cor. \ref{coro1})). The reason can be either the regularity of the function used in this example or the fact that   the values $l^n_0$ and $l^n_1$ taking from Eq. \eqref{equationcotas} are too large.

 \begin{table}[H]
  \begin{center}
   \begin{tabular}{|c|c|c|c|c|c|c|c|} \hline
   $\hat h$   & $\bS$ &  $\bR_{FB}$  &  $\bO_{FB}$ & $\bR_{B}$ &$\bR_{AY}$   &  $\bO_{B}$ & $\bO_{AY}$   \\
   \hline
   $    9.3750e-02$  &$        2.9903$  & $         2.3242 $ & $      2.8853$&$     2.0628$ &$      2.4435$&$    2.8853$&$         2.8853$\\
   $    2.3438e-02$  &$        2.9999$  & $         2.9997 $ & $      2.9999$&$     2.9995$ &$      2.9998$&$    2.9999$&$         2.9999$\\
   $    5.8594e-03$  &$        3.0000$  & $         3.0000 $ & $      3.0000$&$     3.0000$ &$      3.0000$&$    3.0000$&$         3.0000$\\
   $    1.4648e-03$  &$        2.9978$  & $         2.9978 $ & $      2.9978$&$     2.9978$ &$      2.9978$&$    2.9978$&$         2.9978$\\
   \hline
   \end{tabular}
  \end{center}
   \caption{Experiment 1 with $\hat h=\frac342^{-l}$ (non-equally-spaced grid) and estimated orders
     $\log_4(e^{W^n_3}_l/e^{W^n_3}_{l-2})$,    $3\leq l \leq 9$,
$W^n_3=\{i\,:\,0\leq i\leq l^n_0 \,\,\text{or} \,\, l^n_1\leq i\leq 2^{l+2}\}$.}
      \label{tab:a8}
\end{table}

\subsubsection{Experiment 2 with a non-uniform grid}

In order to analyze a function with a strong gradient, in this subsection, we discretize $g(x)$, Eq. \eqref{eq:exp2},  at the interval $[0,2]$ using the non-uniform grid defined by Eq. \eqref{malladonouniforme}.
We take the following window:
\begin{equation*}
\tilde W^n_3=W^n_3\setminus\{l^n_1\}.
\end{equation*}

In Table \ref{tab:a9} we can see that
		the order of the $\bO$ methods is one, due to the presence of the discontinuity, as the hypothesis of Prop. \ref{teo:arandigayanez} are not satisfied. For $\bR$ methods we obtain an improvement in the order but not the optimal. In order to get it we take the following window:
\begin{equation*}
W^n_4=\{i\,:\,0\leq i\leq \tilde l^n_0 \,\,\text{or} \,\, \tilde l^n_1\leq i\leq 2^{l+2}\},
\end{equation*}
with
\begin{equation*}
\begin{split}
&\tilde l^n_0=(i_0-1)+2\log_2(\hat{h})=(2^{l+1}-1)+2\log_2\left(\frac342^{-l}\right),\\
&\tilde l^n_1=(i_0+1)-2\log_2(\hat{h})=(2^{l+1}+1)-2\log_2\left(\frac342^{-l}\right).\\
\end{split}
\end{equation*}

\begin{table}[H]
  \begin{center}
   \begin{tabular}{|c|c|c|c|c|c|c|c|} \hline
   $\hat h$   & $\bS$ &  $\bR_{FB}$  &  $\bO_{FB}$ & $\bR_{B}$ &$\bR_{AY}$   &  $\bO_{B}$ & $\bO_{AY}$   \\
   \hline
    $    9.3750e-02$  &$       0.9750$  & $       1.4998 $ & $       0.9750$&$       1.4742$ &$       1.5462$&$        0.9750$&$    0.9750$\\
    $    2.3438e-02$  &$       1.0791$  & $       2.0050 $ & $       1.0791$&$       2.0231$ &$       2.0569$&$        1.0791$&$    1.0791$\\
    $    5.8594e-03$  &$       1.0823$  & $       2.0514 $ & $       1.0823$&$       2.0539$ &$       2.0647$&$        1.0823$&$    1.0823$\\
    $    1.4648e-03$  &$       1.0827$  & $       2.0749 $ & $       1.0827$&$       2.0757$ &$       2.0786$&$        1.0827$&$    1.0827$\\
   \hline
   \end{tabular}
  \end{center}
   \caption{Experiment 2 with $\hat h=\frac 34 2^{-l}$ (non-equally-spaced grid) and estimated orders
     $\log_4(e^{\tilde W^n_3}_l/e^{\tilde W^n_3}_{l-2})$,    $5\leq l \leq 8$,
$\tilde W^n_3=W^n_3\setminus\{l^n_1\}$.}
      \label{tab:a9}
\end{table}

The results corresponding to this setup are shown in  Table \ref{tab:a10}.
Third order of accuracy is obtained for all methods, in agreement with the theoretical results.

As a conclusion, the kind of reconstructions proposed in the paper allow for replacing the approximations of the derivatives in some points, in order to ensure monotonicity preservation while maintaining optimal order at points that are located at a certain distance  from them.

\begin{table}[H]
  \begin{center}
   \begin{tabular}{|c|c|c|c|c|c|c|c|} \hline
   $\hat h$   & $\bS$ &  $\bR_{FB}$  &  $\bO_{FB}$ & $\bR_{B}$ &$\bR_{AY}$  &  $\bO_{B}$ & $\bO_{AY}$   \\
   \hline
    $    4.6875e-02$ &$           3.1187$  & $        3.1856  $   &   $    3.1187 $ & $     3.1210$&$    3.1391$ &$      3.1187$&$      3.1187$\\
    $    1.1719e-02$ &$           3.1291$  & $        3.0872  $   &   $    3.1291 $ & $     3.0615$&$    3.0616$ &$      3.1291$&$      3.1291$\\
    $    2.9297e-03$ &$           3.1248$  & $        3.0209  $   &   $    3.1248 $ & $     3.0143$&$    3.0139$ &$      3.1248$&$      3.1248$\\
    $    7.3242e-04$ &$           3.1166$  & $        2.9547  $   &   $    3.1166 $ & $     2.9530$&$    2.9528$ &$      3.1166$&$      3.1166$\\
   \hline
   \end{tabular}
  \end{center}
   \caption{Experiment 2 with $\hat h=\frac342^{-l}$ (non-equally-spaced grid) and estimated orders
     $\log_4(e^{ W^n_4}_l/e^{ W^n_4}_{l-2})$,    $4\leq l \leq 10$,
$W^n_4=\{i\,:\,0\leq i\leq \tilde l^n_0 \,\,\text{or} \,\, \tilde l^n_1\leq i\leq 2^{l+2}\}$.}
      \label{tab:a10}
\end{table}

\subsection{Monotonicity}\label{sec:monotonicity}
In this section we show an example, indicated as experiment 3 (Table \ref{tab:ex12}), in which we compare the reconstructions obtained with the different methods considered in this paper.
These data have been used in  \cite{FrCar}.
In this example  we only know the values of the function at certain nodes.
The discretization is not equally spaced and, therefore,   methods $B$ and
$AY$ are different.
\begin{table}[H]
  \begin{center}
\begin{tabular}{|l||r|r|r|r|r|r|r|r|r|}
  \hline
  $i$   & 1 & 2 & 3 & 4 & 5 & 6 & 7 & 8 & 9 \\\hline
  $x_i$ & $7.99$ & $8.09$ & $8.19$ & $8.7$ & $9.2$ & $10$ & $15$ & $12$ & $20$ \\
  $y_i$ & $0$ & $2.76429e{-}5$ & $4.37498e{-}2$ & $ 0.169183$ & $ 0.469428$ & $ 0.943740$ & $ 0.998636$ & $0.999919$ & $ 0.999994 $ \\
  \hline
\end{tabular}
    \caption{Data for the experiment 3}
    \label{tab:ex12}
  \end{center}
\end{table}



The algorithm modifies the values at the nodes
$i=2,6,7,8$. In this case, we obtain monotone reconstruction in the complete interval (Figure \ref{fig:exp4}). The results are very similar in all cases, i.e., both methods produce similar monotone curves because the values which are modified are the most relevant in this example.

\begin{figure}[!bpt]
\begin{center}
\begin{tabular}{cccc}
\includegraphics[width=4.0cm,height=3.5cm]{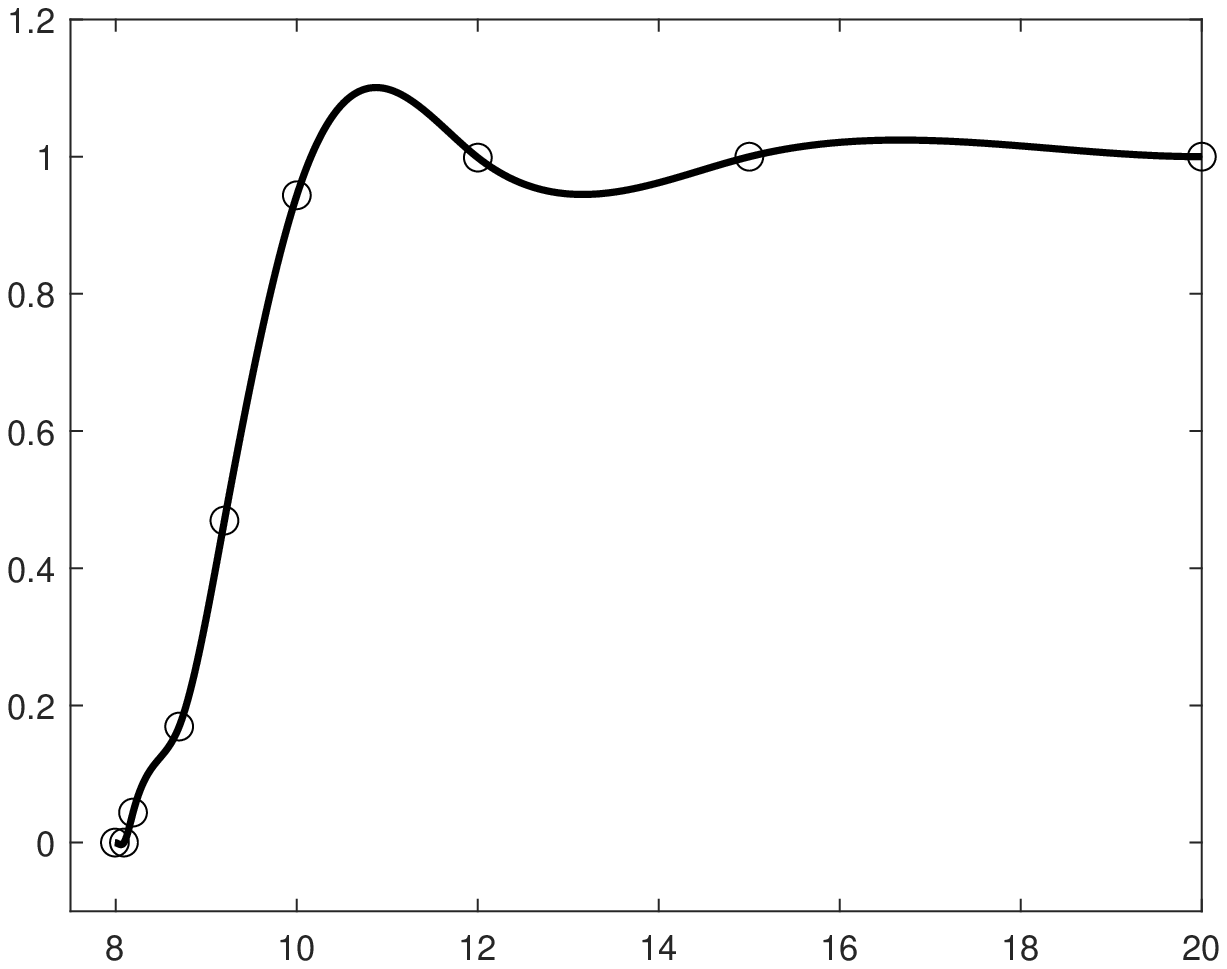}&\includegraphics[width=4.0cm,height=3.5cm]{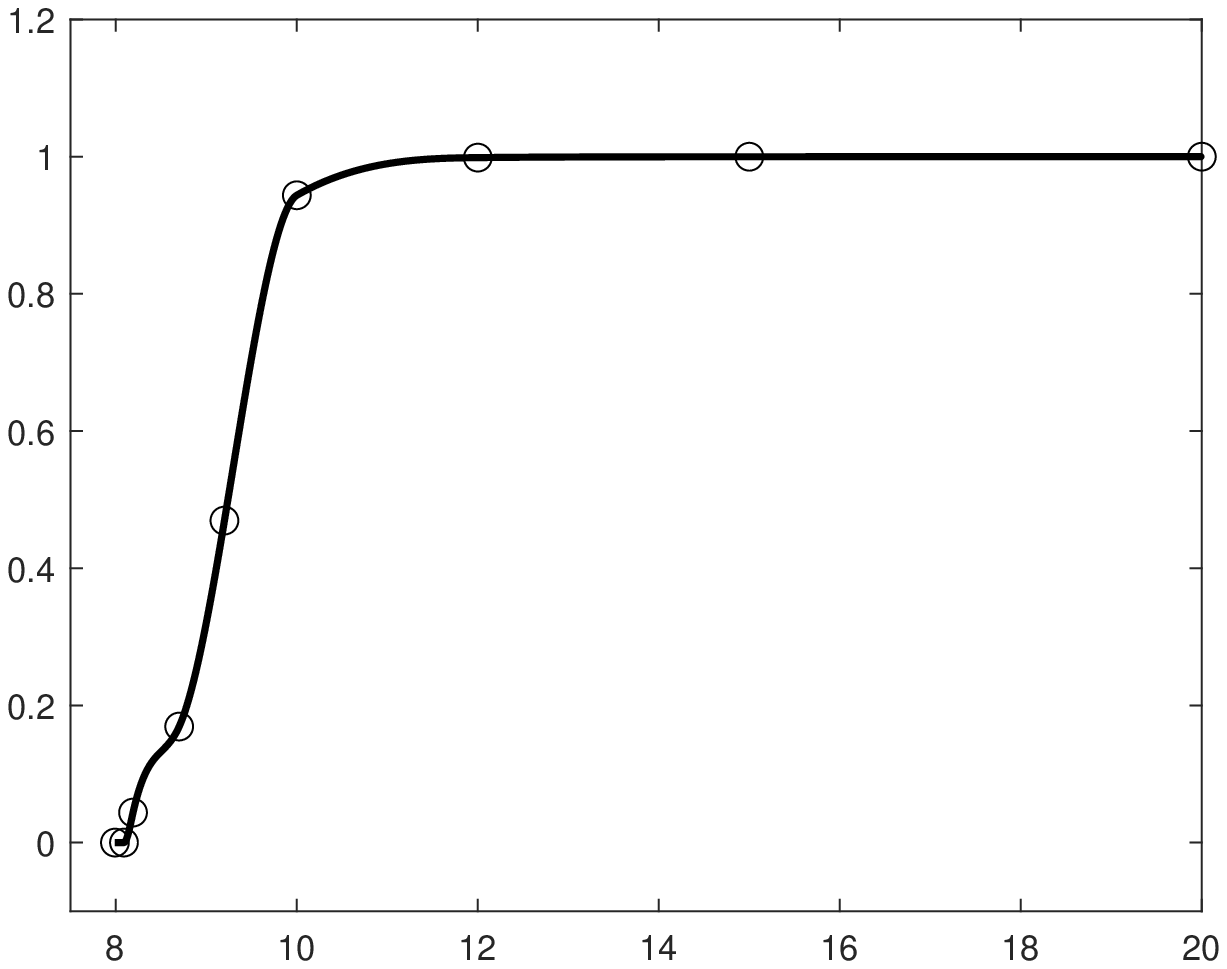} &\includegraphics[width=4.0cm,height=3.5cm]{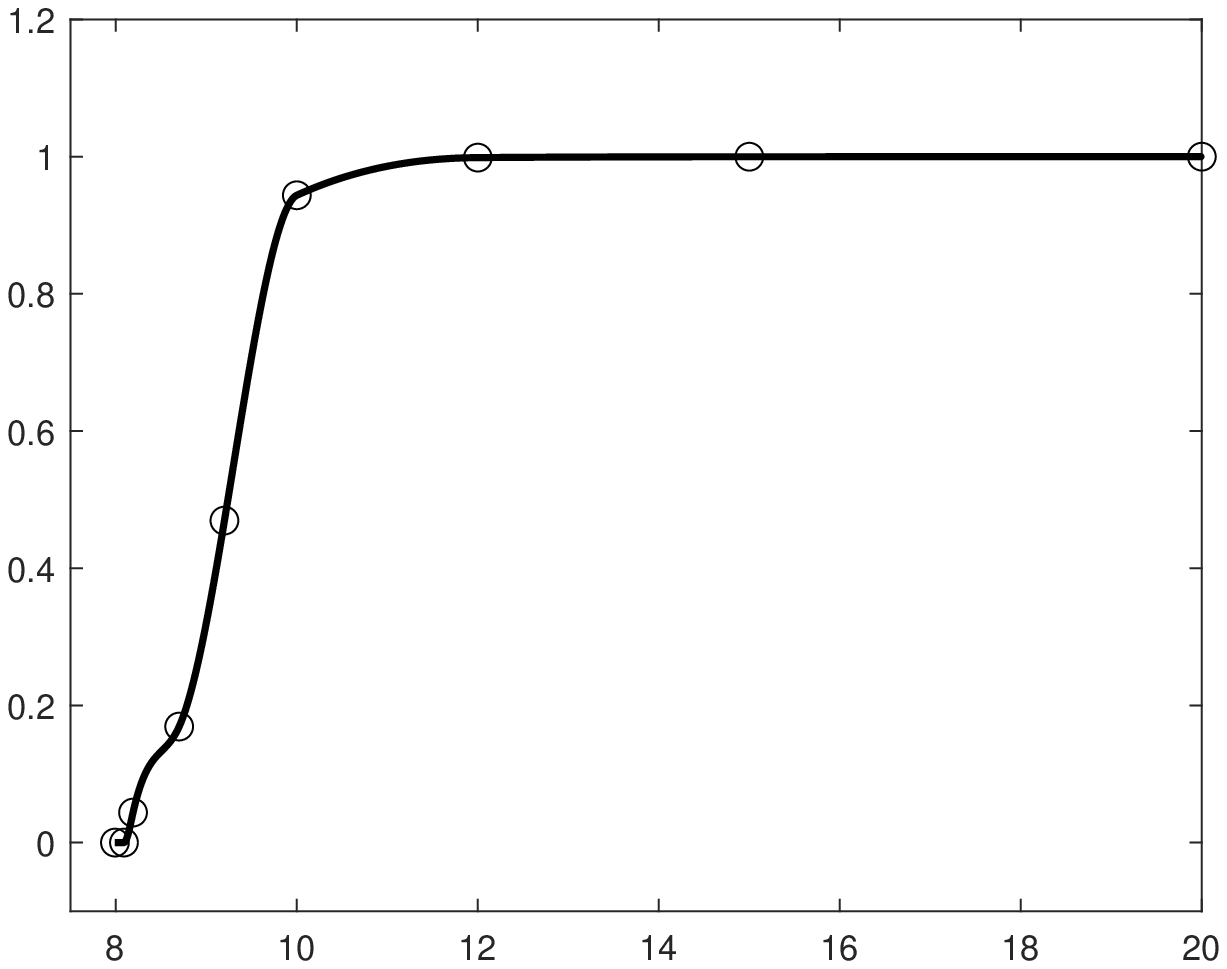} & \includegraphics[width=4.0cm,height=3.5cm]{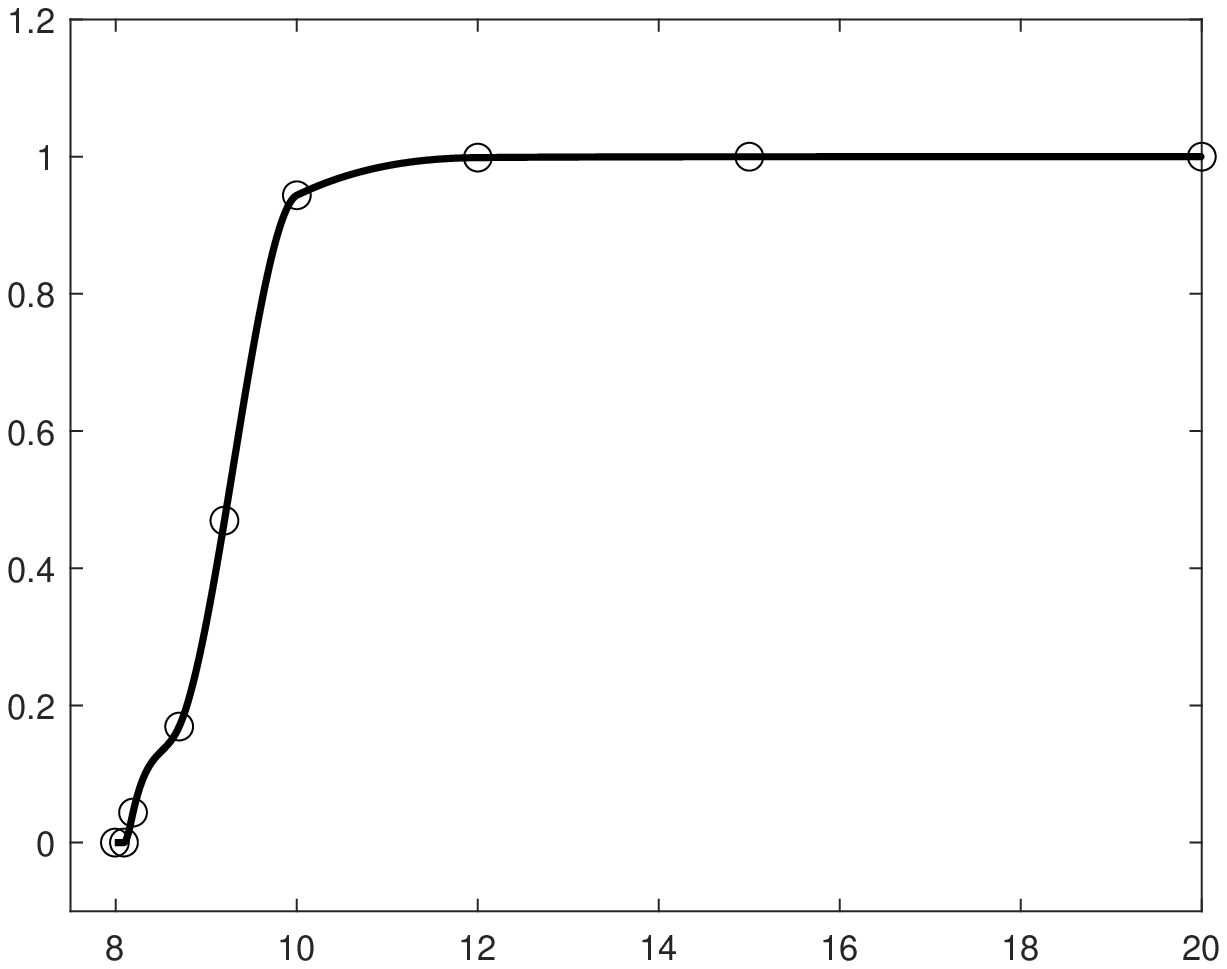} \\
$\bS$& $\bR_{FB}$ & $\bR_{B}$ &$\bR_{AY}$\\
\end{tabular}
\end{center}
\caption{Reconstructions obtained for Experiment 3 using the different methods}
    \label{fig:exp4}
 \end{figure}

\section{Conclusions}\label{sec:conclusion}

In this work, we have introduced two new algorithms to obtain monotone cubic spline
	interpolants. We have considered the case in which there exists a discontinuity or a high gradient in the data, and considered two options, based on modifying the approximation to the derivatives in the points where monotonicity constraints are violated.
	In the first one we rewrite the spline system fixing the modified derivatives and recompute the derivatives using the modified spline system at both sides. Using this algorithm the regularity is $C^2$ in all points except at the ones where the derivative approximation was modified, and the order is reduced in a neighborhood of these, but it is conserved in the rest of the interval. In the second algorithm we again modify the approximations of the derivatives where required but the rest of the values are kept as initially computed. In this case we conserve the order in all points but the regularity is lost in a neighborhood of the points where the derivative was modified. Some numerical tests confirm these results.

\bibliographystyle{these}

\end{document}